\documentclass[a4paper, english, 12pt, reqno, dvipsnames]{amsart}

\usepackage{amssymb}
\usepackage[english]{babel}
\usepackage[utf8]{inputenc}
\usepackage{tikz,tikzscale}
\usepackage{tikz,tikzscale}
\usetikzlibrary{calc}
\tikzset{>=latex}
\usepackage{enumerate,enumitem}
\usepackage{booktabs,multirow}
\usepackage{amsfonts}
\usepackage{dsfont}

\newtheorem{theorem}{Theorem}[section]
\newtheorem{lemma}[theorem]{Lemma}

\theoremstyle{definition}

\theoremstyle{remark}
\newtheorem{remark}[theorem]{Remark}

\numberwithin{equation}{section}

\newcommand{\elem}{\ensuremath{T}}
\newcommand{\mesh}{\ensuremath{\mathcal T}}
\newcommand{\edge}{\ensuremath{\mathcal E}}
\newcommand{\vertex}{\ensuremath{\vec a}}
\newcommand{\transposed}{\ensuremath{\dagger}}

\newcommand{\faceSet}{\ensuremath{\mathcal F}}
\newcommand{\faceSetDir}{\ensuremath{\mathcal F^\textup D}}
\newcommand{\face}{\ensuremath{F}}
\newcommand{\skeleton}{\ensuremath{\Sigma}}
\newcommand{\skeletalSpace}{\ensuremath{M}}
\newcommand{\contElementSpace}{\ensuremath{V^\textup c}}

\newcommand{\linElementSpace}{\ensuremath{\overline V^\textup c}}
\newcommand{\discElementSpace}{\ensuremath{V}}
\newcommand{\polynomials}{\ensuremath{\mathcal P}}
\newcommand{\level}{\ensuremath{\ell}}
\newcommand{\iterMgInner}{m}

\newcommand{\Div}{\nabla\!\cdot\!}

\newcommand{\avgOp}{\ensuremath{I^\textup{avg}}}
\newcommand{\linearInterpolation}{\overline{I}}

\newcommand{\injectionOp}{\ensuremath{I}}
\newcommand{\projectionOp}{\ensuremath{P}}

\newcommand{\contLinProj}{\ensuremath{\overline \Pi^\textup c}}

\newcommand{\liftingOp}{\ensuremath{S}}

\renewcommand{\vec}[1]{\ensuremath{\boldsymbol{#1}}}
\newcommand{\Nu}{\ensuremath{\vec \nu}}
\newcommand{\dx}{\ensuremath{\, \textup d x}}
\newcommand{\ds}{\ensuremath{\, \textup d \sigma}}

\newcommand{\avg}[1]{\{\!\{ #1 \}\!\}}

\newcommand{\llangle}{\ensuremath{\langle \! \langle}}
\newcommand{\rrangle}{\ensuremath{\rangle \! \rangle}}
\newcommand{\nnorm}{\ensuremath{\vert \! \vert \! \vert}}

\newcommand{\gradU}{\ensuremath{L}}
\newcommand{\IR}{\ensuremath{\mathbb R}}
\newcommand{\IN}{\ensuremath{\mathbb N}}
\newcommand{\timeStep}{{\ensuremath{\Delta t}}}
\newcommand{\stab}{\ensuremath{S}}
\newcommand{\tol}{\ensuremath{\epsilon_\textup{tol}}}

\usepackage[colorlinks = true, linkcolor = blue, citecolor = blue, urlcolor = blue]{hyperref}

\makeatletter
\newcommand\footnoteref[1]{\protected@xdef\@thefnmark{\ref{#1}}\@footnotemark}
\makeatother

\begin{document}

\title[Multigrid for Stokes HDG]{Homogeneous multigrid method for HDG applied to the Stokes equation} 

\author{Peipei Lu}
\address{Department of Mathematics Sciences, Soochow University, Suzhou, 215006, China}
\email{pplu@suda.edu.cn}

\author{Wei Wang}
\address{School of Mathematics, University of Minnesota, 206 Church St SE, Minneapolis, MN, 55455, USA}
\email{wang9585@umn.edu}

\author{Guido Kanschat}
\address{Interdisciplinary Center for Scientific Computing (IWR), Heidelberg University, Mathematikon, Im Neuenheimer Feld 205, 69120 Heidelberg, Germany}
\email{kanschat@uni-heidelberg.de}

\author{Andreas Rupp}
\address{School of Engineering Science, Lappeenranta--Lahti University of Technology, P.O. Box 20, 53851 Lappeenranta, Finland}
\email{andreas.rupp@fau.de}


\begin{abstract}
 We propose a multigrid method to solve the linear system of equations arising from a hybrid discontinuous Galerkin (in particular, a single face hybridizable, a hybrid Raviart--Thomas, or a hybrid Brezzi--Douglas--Marini) discretization of a Stokes problem. Our analysis is centered around the augmented Lagrangian approach and we prove uniform convergence in this setting. Numerical experiments underline our analytical findings.
 \\[1ex] \noindent \textsc{Keywords.}
 Augmented Lagrangian approach, hybrid discontinuous Galerkin, multigrid method, Stokes equation.
\end{abstract}

\date{\today}
\maketitle
\section{Introduction}
Hybrid discontinuous Galerkin (HDG) methods have been a very active field of research in the last years, and they have been applied to many different partial differential equations (PDEs). For Stokes problems, one advantage of HDG schemes is that they may provide exactly divergence-free velocities especially for the three dimensional case through postprocessing. Cockburn et al. proposed and analysed several HDG methods for the Stokes Flow. The first HDG method for the Stokes equation was introduced in \cite{CockburnG09} for the velocity\--pressure\--vorticity formulation. Another two HDG methods are based on the velocity-pressure-gradient formulation \cite{NguyenPC10} and the velocity-pressure-stress formulation \cite{CockburnNP10} of the Stokes system respectively. In \cite{CockburnNP10} it is shown that all of the above HDG methods share the following common features: First, when all the components of the approximate solution are discretized by the polynomials of degree $k$, convergence with the optimal order of $k+1$ in $L^2$ for any $k \geq 0 $ is achieved. Second, a globally divergence-free velocity can be computed by using an element-by-element postprocessing of the approximation velocity. A pressure-robust HDG method was introduced in \cite{RhebergenW17} for the Stokes equations, which is also locally mass conserving, momentum conserving and energy stable. The generalization to the Navier–Stokes equations can be found in  \cite{RhebergenW18}. Lehrenfeld presented a symmertric interior penalty, divergence-conforming HDG method with projected jumps in \cite{LehrenfeldPhD}.
 
However, constructing multigrid methods for HDG discretizations has been a major issue, since many intuitive choices for injection operators have turned out to be instable, even for the Poisson equation \cite{TanPhD}. HDG multigrid methods for the diffusion equation have been devised by first transferring the HDG approximation to a continuous finite element approximation and performing the multigrid method for this method \cite{CockburnDGT2013}. Later, Lu, Rupp, and Kanschat proposed a multigrid framework for the Poisson equation that does not internally change the approximation scheme and called this approach `homogeneous multigrid', see \cite{LuRK21,LuRK21a,LuRK22b}. For the Stokes equations, the design of robust solver for condensed HDG discretizations is more challenging. A multigrid method for such a task has been proposed for lowest order HDG and `projected jumps' in \cite{FuK22}, which exploits an an equivalence between such a HDG scheme and Crouzeix--Raviart finite elements. Afterwards, \cite{FuK22} uses the Crouzeix--Raviart mutigrid theory established by Brenner, see e.g. \cite{Brenner04Vcycle}. The authors also proposed a uniform block-diagonal preconditioner for condensed, $H(\operatorname{div})$ conforming HDG schemes applied parameter-dependent saddle point problems, including the generalized Stokes equations and the linear elasticity equations in \cite{FuKw22}. A new preconditioner was introduced in \cite{RhebergenW22}
for a recently developed pressure-robust HDG scheme \cite{RhebergenW17}, by showing the Schur complement of the statically condensed system is spectrally equivalent to a simple trace pressure mass matrix, it is proven to be optimal.

In this paper, we transfer the results in \cite{LuRK21,LuRK21a,LuRK22b} to a Stokes problem, which may be discretized by one of the following HDG methods: the single face hybridizable (SFH), the hybrid Raviart--Thomas (RT-H), or the hybrid Brezzi--Douglas--Marini (BDM-H) method. Combing the augmented Lagrangian method, the globally coupled unknowns can be reduced to the numerical trace of the velocity only. We focus on the multigrid method for the corresponding HDG system. Using the injection operators proposed in \cite{LuRK21}, we prove the uniform convergence of the V-cycle multigrid method for the augmented Lagrangian approach.

The remainder of this manuscript is structured as follows: In Section \ref{SEC:method}, we describe the problem and its HDG discretization. Afterwards, in Section \ref{SEC:connections}, we discuss the relations between the approximate solutions of the RT-H, the BDM-H, and  the SFH methods. Next, we analyze the local solvers in some details and construct an auxiliary problem, for which we prove an error estimate. The last steps of our analysis contain an investigation of injection operators in Section \ref{SEC:injection_op} and the illustration of convergence results for our multigrid scheme. A section on numerical results and short conlcusions wrap up our manuscript.
%
\section{Description of the used method}\label{SEC:method}
%
We consider the velocity--pressure--gradient formulation of the Stokes equation in a bounded, Lipschitz domain $\Omega \subset \IR^3$. According to \cite{CockburnNP10}, this formulation is beneficial for HDG methods, since if based on this formulation, HDG methods provide the qualitatively best approximations to the solution of the Stokes equation as compared to other methods of similar computational complexity:
\begin{subequations}\label{EQ:stokes_mixed}\begin{align}
 \gradU - \nabla u & = 0 && \text{ in } \Omega, \\
 - \Div \gradU + \nabla p & = f && \text{ in } \Omega, \\
 \Div u & = 0 && \text{ in } \Omega, \\
 u & = 0 && \text{ on } \partial \Omega,
\end{align}\end{subequations}
where the Dirichlet boundary condition $g$ is set to be zero for the ease of representation. Note that this choice suffices $\int_{\partial \Omega} g \cdot \Nu \ds = 0$.

Convergence analysis of HDG methods based on \eqref{EQ:stokes_mixed} is presented in \cite{CockburnGNPS11}. This analysis also includes the single-face hybridizable (SFH) method. It shows that pressure, velocity, and the gradient of the velocity converge with the optimal orders $k+1$ (where $k$ denotes the order of the used local polynomial spaces) with respect to the $L^2$ norms. Moreover, using an element-by-element postprocessing scheme, an improved velocity approximation can be obtained. This approximation is pointwise divergence-free, $H(\operatorname{div})$ conforming, and converges with an order of $k+2$ if $k \ge 1$.

One of the main advantages of HDG over the DG method is that it reduces the globally coupled unknowns to the numerical trace of the velocity and the mean of the pressure on element faces \cite{NguyenPC10}. This
leads to a significant reduction in the size of the resulting global matrix. Moreover, by using the augmented Lagrangian method, the globally coupled unknowns are further reduced to the numerical trace of the velocity only.
%
\subsection{Augmented Lagrangian approach}
%
An augmented Lagrangian method is based on an evolution problem whose limit, for the time going to infinity, is the solution of the original problem. That is, for $t > 0$
\begin{subequations}\label{EQ:stokes_parab}\begin{align}
 \gradU(t) - \nabla u(t) & = 0 && \text{ in } \Omega, \\
 - \Div \gradU(t) + \nabla p(t) & = f && \text{ in } \Omega, \\
 \partial_t p(t) + \Div u(t) & = 0 && \text{ in } \Omega, \\
 u(t) & = 0 && \text{ on } \partial \Omega, \\
 p(0) & = p_0 && \text{ in } \Omega.
\end{align}\end{subequations}

The augmented approach consists of an iterative method, which is obtained using backward Euler in time and HDG in space.
%
\subsection{Preliminaries and notation}
%
To this end, we start with a successively refined series of simplicial meshes $\mesh_\level$, $\level \in \IN$. We assume all of its meshes to be regular (i.e., no elements are anisotropic or otherwise disorted), geometrically conforming (i.e., each face of a cell is either the face of one other cell or part of the boundary), and assume that the refinement is not too fast, such that there is a constant $c_\textup{ref} > 0$ such that
\begin{equation}
 h_\level \ge c_\textup{ref} h_{\level-1}.
\end{equation}
The set of faces of $\mesh_\level$ is denoted by $\faceSet_\level$. The skeleton of $\mesh_\level$ is the union of all faces $\skeleton_\level := \bigcup_{\face \in \faceSet_\level} \face$. We define the space of piecewise polynomials of degree at most $p$ on the skeleton as
\begin{equation}
 \skeletalSpace_\level := \left\{ \lambda \in [L^2 (\skeleton_\level)]^3 \;\middle|\;
 \begin{array}{r@{\,}c@{\,}ll}
  \lambda_{|\face} &\in& [\polynomials_p (\face)]^3 & \forall \face \in \faceSet_\level\\
  \lambda_{|\face} &=& 0 & \forall \face \in \faceSetDir_\level    
 \end{array}\right\}.
\end{equation}
Moreover, we denote by $W_\level$ the approximation space of $\gradU(t)$, by $V_\level$ the approximation space of $u(t)$, and by $Q_\level$ the approximation space of $p(t)$, respectively.

Beyond this, we define a scalar product and a norm for square integrable function in the bulk domain $\Omega$ via
\begin{equation}
 (u, v) := \int_\Omega u v \dx \qquad \text{ and } \qquad \| u \|^2_0 := \sqrt{(u,u)}, 
\end{equation}
for all $u, v \in L^2(\Omega)$. If the respective functions $u$ and $v$ are assumed to be vector-valued or matrix-valued, the used product needs to be adapted accordingly. On the skeleton, we consider two different scalar products with associated norms: The first one is defined as
\begin{equation}
 \llangle \lambda, \mu \rrangle_\level = \sum_{\elem \in \mesh_\level} \int_{\partial \elem} \lambda \cdot \mu\ds,
\end{equation}
for $\lambda, \, \mu \in [L^2(\skeleton_\level)]^3$. Note that its value grows if the mesh is refined, and that interior faces appear twice in this definition such that expressions like $\llangle u, \mu \rrangle_\level$ with possibly discontinuous $u|_{\elem} \in [H^1(\elem)]^3$ for all $\elem \in \mesh_\level$ are defined without further ado. Its induced norm is denoted as $\nnorm \cdot \nnorm_\level$.

The second one is commensurate with the $L^2$-inner product in the bulk domain, namely
\begin{equation}
 \langle \lambda, \mu \rangle_\level = \sum_{\elem \in \mesh_\level} \frac{|\elem|}{|\partial \elem|} \int_{\partial \elem} \lambda \cdot \mu \ds \cong \sum_{\face \in \faceSet_\level} h_\face  \int_{\face} \lambda \cdot \mu \ds.
\end{equation}
Its induced norm is $ \| \mu \|^2_\level = \langle \mu, \mu \rangle_\level$, which can naturally be extended to functions that have discontinuities between two mesh cells.

Assuming that $p_0$ is an initial (continuous) guess for the pressure, it is projected to its respective discrete space using the $L^2$ projection, which is characterized via
\begin{equation}
 (p^0_\level, q) = (p_0,q) \qquad \text{ for all } q \in Q_\level.
\end{equation}
%
\subsection{Local solvers}
%
Next, given a constant time step $\timeStep$ and a pressure $p^{n-1}_\level$ we define the iterate $(\gradU^n_\level, u^n_\level, p^n_\level, \lambda^n_\level) \in W_\level \times V_\level \times Q_\level \times \skeletalSpace_\level$ as an approximation to $\gradU(n \timeStep)$, $u(n \timeStep)$, $p(n \timeStep)$, and $u(n \timeStep)$, respectively, by
\begin{subequations}\label{EQ:scheme}\begin{align}
 (\gradU^n_\level, G) + (u^n_\level, \Div G) - \llangle \lambda^n_\level, G  \Nu \rrangle_\level & = 0, \\
 (- \Div \gradU^n_\level + \nabla p^n_\level, v) + \llangle \stab (u^n_\level - \lambda^n_\level), v \rrangle_\level & = (f, v), \\
 \tfrac{1}{\timeStep} (p^n_\level, q) - (u^n_\level, \nabla q) + \llangle \lambda^n_\level \cdot \Nu, q \rrangle_\level & = \tfrac{1}{\timeStep} (p^{n-1}_\level, q), \\
 \llangle - \hat{\gradU}^n_\level \Nu + \hat p^n_\level \Nu, \mu \rrangle_\level & = 0
\end{align}\end{subequations}
for all $(G, v, q, \mu) \in W_\level \times V_\level \times Q_\level \times \skeletalSpace_\level$. Here, the numerical flux is defined as
\begin{equation}
 \hat{\gradU}^n_\level \Nu + \hat p^n_\level \Nu = \gradU^n_\level \Nu + p^n_\level \Nu + \stab (u^n_\level - \lambda^n_\level).
\end{equation}

In practice, the iteration can be stopped, when the relative error of the pressure is less than a prescribed tolerance $\tol$, see \cite{NguyenPC10}. That is, we stop at $n = n_\textup{iter}$ if
\begin{equation}\label{EQ:sol_accuracy}
 \frac{ \| p^{n_\textup{iter}}_\level - p^{n_\textup{iter} -1 }_\level \|_0 }{ \| p^{n_\textup{iter}}_\level \|_0 } < \tol.
\end{equation}

The numerical results of \cite{NguyenPC10} show that $n_\textup{iter}$ is independent of the mesh size/level and independent of the polynomial degrees of the test and trial spaces. Hence, the augmented Lagrangian approach appears to be attractive for solving the HDG discretization of the Sokes equation.

Next, we define the operators mapping $\lambda$ to the respective solutions of the element--local problem
\begin{equation}
 \skeletalSpace_\level \ni \lambda \mapsto (\gradU^\timeStep_\level \lambda, u^\timeStep_\level \lambda, p^\timeStep_\level \lambda) \in W_\level \times V_\level \times Q_\level
\end{equation}
by claiming that
\begin{subequations}\label{EQ:local_sol}\begin{align}
 (\gradU^\timeStep_\level \lambda, G)_\elem + (u^\timeStep_\level \lambda, \Div G)_\elem & = \llangle \lambda, G \Nu \rrangle_{\partial \elem}, \label{EQ:local_sol_gradu} \\
 (- \Div \gradU^\timeStep_\level \lambda + \nabla p^\timeStep_\level \lambda, v)_\elem + \llangle \stab (u^\timeStep_\level \lambda - \lambda), v \rrangle_{\partial \elem} & = 0, \label{EQ:local_sol_rhs_zero}\\
 \tfrac{1}{\timeStep} (p^\timeStep_\level \lambda, q)_\elem - (u^\timeStep_\level \lambda, \nabla q)_\elem & = - \llangle \lambda \cdot \Nu, q \rrangle_{\partial \elem}\label{EQ:local_sol_bal}
\end{align}\end{subequations}
holds for all $(G, v, q) \in W_\level \times V_\level \times Q_\level$. In the same way, we can define element--local solution operators for the right hand side
\begin{equation}
 L^2(\Omega) \ni f \mapsto (\gradU^\timeStep_\level f, u^\timeStep_\level f, p^\timeStep_\level f) \in W_\level \times V_\level \times Q_\level
\end{equation}
via the assumption that
\begin{subequations}\label{EQ:local_sol_f}\begin{align}
 (\gradU^\timeStep_\level f, G)_\elem + (u^\timeStep_\level f, \Div G)_\elem & = 0, \label{EQ:local_sol_f_gradu}\\
 (- \Div \gradU^\timeStep_\level f + \nabla p^\timeStep_\level f, v)_\elem + \llangle \stab u^\timeStep_\level f, v \rrangle_{\partial \elem} & = (f, v), \\
 \tfrac{1}{\timeStep} (p^\timeStep_\level f, q)_\elem - (u^\timeStep_\level f, \nabla q)_\elem & = 0
\end{align}\end{subequations}
need to hold for all $(G, v, q) \in W_\level \times V_\level \times Q_\level$. Beyond that, we can evolve the solution components, if we know the pressure of the previous time step---which is therefore denoted by $m$---inducing the mapping
\begin{equation}
 Q_\level \ni m \mapsto (\gradU^\timeStep_\level m, u^\timeStep_\level m, p^\timeStep_\level m) \in W_\level \times V_\level \times Q_\level
\end{equation}
which are defined such that
\begin{subequations}\label{EQ:local_sol_m}\begin{align}
 (\gradU^\timeStep_\level m, G)_\elem + (u^\timeStep_\level m, \Div G)_\elem & = 0, \\
 (- \Div \gradU^\timeStep_\level m + \nabla p^\timeStep_\level m, v)_\elem + \llangle \stab u^\timeStep_\level m, v \rrangle_{\partial \elem} & = 0, \\
 \tfrac{1}{\timeStep} (p^\timeStep_\level m, q)_\elem - (u^\timeStep_\level m, \nabla q)_\elem & = \tfrac{1}{\timeStep} (m, q)_\elem
\end{align}\end{subequations}
holds for all $(G, v, q) \in W_\level \times V_\level \times Q_\level$.
%
\subsection{General remarks}
%
The $\lambda^n_\level$ in \eqref{EQ:scheme} admit the global equality, cf.\ \cite[Thm.\ 3.1]{NguyenPC10},
\begin{equation}\label{EQ:hdg_condensed}
 a^\timeStep_\level(\lambda^n_\level, \mu) = (f, u^\timeStep_\level \mu) - \tfrac{1}{\timeStep} (p^{n-1}_\level, p^\timeStep_\level \mu) \qquad \forall \mu \in \skeletalSpace_\level
\end{equation}
where
\begin{equation}\label{EQ:bilinear}
 a^\timeStep_\level(\lambda, \mu) = (\gradU^\timeStep_\level \mu, \gradU^\timeStep_\level \mu) + \llangle \stab (u^\timeStep_\level \lambda - \lambda), (u^\timeStep_\level \mu - \mu ) \rrangle_\level + \tfrac{1}{\timeStep} (p^\timeStep_\level \lambda, p^\timeStep_\level \mu).
\end{equation}

In this paper, we focus on the multigrid method for \eqref{EQ:hdg_condensed} and restrict ourselves to considering the SFH method. That is,
\begin{equation}
 \stab = \tau_\level \begin{pmatrix} 1 & 0 & 0 \\ 0 & 1 & 0 \\ 0 & 0 & 1 \end{pmatrix}, \quad \text{ with } \quad \tau_\level = \begin{cases} \tau^\star_\level & \text{ on } \face^\star_\elem \\ 0 & \text{ on } \partial \elem \setminus \face^\star_\elem, \end{cases}
\end{equation}
where $\face^\star_\elem$ is an arbitrary face of $\elem$. However, our analysis also covers the RT-H and BDM-H methods. The following list contains the respective choices of approximation spaces and stabilization parameters to obtain the SFH, RT-H, and BDM-H methods:
\begin{itemize}[leftmargin=*]
 \item SFH: $W_\elem = [\polynomials_p(\elem)]^{3 \times 3}$, $V_\elem = [\polynomials_p(\elem)]^3$, $Q_\elem = \polynomials_p(\elem)$, $\tau^\star_\level > 0$,
 \item RT-H: $W_\elem = [\text{Raviart--Thomas space of degree }p]^3$, $V_\elem = [\polynomials_p(\elem)]^3$, $Q_\elem = \polynomials_p(\elem)$, $\tau^\star_\level = 0$,
 \item BDM-H: $W_\elem = [\polynomials_p(\elem)]^{3 \times 3}$, $V_\elem = [\polynomials_{p-1}(\elem)]^3$, $Q_\elem = \polynomials_p(\elem)$, $\tau^\star_\level = 0$.
\end{itemize}
 
Notably, bilinear form $a^\timeStep_\level$ induces the norm $\| \cdot \|_{a_\level}$.
\section{Relations among RT-H, BDM-H, and SFH}\label{SEC:connections}
%
Let us define $\tilde u^\timeStep_\level \lambda \in [\polynomials_p(\elem)]^3$ by
\begin{subequations}\label{EQ:const_tilde_u}\begin{align}
 \tilde u^\timeStep_\level \lambda & = \lambda && \text{ on } \face^\star_\level, \label{EQ:const_tilde_u_face} \\
 (\tilde u^\timeStep_\level \lambda, v)_\elem & = ( u^\timeStep_{\level,\textup{BDM}} \lambda, v)_\elem && \text{ for all } v \in [\polynomials_{p-1}(\elem)]^3. \label{EQ:const_tilde_u_elem}
\end{align}\end{subequations}

We will show that $(\gradU^\timeStep_{\level,\textup{BDM}} \lambda, \tilde u^\timeStep_\level \lambda, p^\timeStep_{\level,\textup{BDM}} \lambda)$ is the solution of \eqref{EQ:local_sol} for SFH in the following.
\begin{lemma}\label{LEM:sfh_bdm_loc}
 For all $\lambda \in \skeletalSpace_\level$, we have that
 \begin{align*}
  \gradU^\timeStep_{\level,\textup{BDM}} \lambda & = \gradU^\timeStep_{\level,\textup{SFH}} \lambda, &
  p^\timeStep_{\level,\textup{BDM}} \lambda & = p^\timeStep_{\level,\textup{SFH}} \lambda, \\
  u^\timeStep_{\level,\textup{SFH}} \lambda & = \lambda \text{ on } \face^\star_\elem, &
  u^\timeStep_{\level,\textup{BDM}} \lambda & = \Pi_{\level,p-1} u^\timeStep_{\level,\textup{SFH}} \lambda,
 \end{align*}
 where $\Pi_{\level,p-1}$ is the element-wise $L^2$ projection to $[\polynomials_{p-1}(\elem)]^3$.
\end{lemma}
\begin{proof}
 Using \eqref{EQ:const_tilde_u_elem} and \eqref{EQ:local_sol_gradu} we obtain for any $G \in [\polynomials_p(\elem)]^{3 \times 3}$ that
 \begin{equation}
  (\gradU^\timeStep_{\level,\textup{BDM}} \lambda, G)_\elem + (\tilde u^\timeStep_\level \lambda, \Div G)_\elem = \llangle \lambda, G  \Nu \rrangle_{\partial \elem},
 \end{equation}
 and from \eqref{EQ:local_sol_rhs_zero}, we know that
 \begin{equation}
  (- \Div \gradU^\timeStep_{\level, \textup{BDM}} \lambda + \nabla p^\timeStep_{\level, \textup{BDM}} \lambda, v)_\elem = 0 \qquad \forall v \in [\polynomials_{p-1}(\elem)]^3.
 \end{equation}
 Since $- \Div \gradU^\timeStep_{\level, \textup{BDM}} \lambda + \nabla p^\timeStep_{\level, \textup{BDM}} \lambda \in  [\polynomials_{p-1}(\elem)]^3$, we can deduce that $- \Div \gradU^\timeStep_{\level, \textup{BDM}} \lambda + \nabla p^\timeStep_{\level, \textup{BDM}} \lambda = 0$, which combined with \eqref{EQ:const_tilde_u_face} results in the observation that
 \begin{equation}
  (- \Div \gradU^\timeStep_{\level, \textup{BDM}} \lambda + \nabla p^\timeStep_{\level, \textup{BDM}} \lambda, v) + \llangle \stab (\tilde u^\timeStep_\level \lambda - \lambda), v  \rrangle_{\partial \elem} = 0 \quad \forall v \in [\polynomials_p(\elem)]^3.
 \end{equation}
 Finally, using \eqref{EQ:local_sol_bal} and \eqref{EQ:const_tilde_u_elem}, we obtain
 \begin{equation}
  \tfrac{1}{\timeStep} (p^\timeStep_{\level,\textup{BDM}} \lambda, q)_\elem - (\tilde u^\timeStep_\level \lambda, \nabla q)_\elem = - \llangle \lambda \cdot \Nu, q \rrangle_{\partial \elem} \qquad \forall q \in \polynomials_p(\elem),
 \end{equation}
 which indicates that $(\gradU^\timeStep_{\level,\textup{BDM}} \lambda, \tilde u^\timeStep_\level \lambda, p^\timeStep_{\level,\textup{BDM}} \lambda)$ is the solution of \eqref{EQ:local_sol} for SFH and hence Lemma \ref{LEM:sfh_bdm_loc} holds
\end{proof}
\begin{lemma}\label{LEM:rt_bdm_loc}
 For all $\lambda \in \skeletalSpace_\level$, we have that
 \begin{gather*}
  \gradU^\timeStep_{\level,\textup{BDM}} \lambda = \gradU^\timeStep_{\level,\textup{RT}} \lambda, \qquad \qquad
  p^\timeStep_{\level,\textup{BDM}} \lambda = p^\timeStep_{\level,\textup{RT}} \lambda, \\
  u^\timeStep_{\level,\textup{BDM}} \lambda = \Pi_{\level,p-1} u^\timeStep_{\level,\textup{RT}} \lambda.
 \end{gather*}

\end{lemma}
\begin{proof}
 For RT, \eqref{EQ:local_sol_rhs_zero} gives us that
 \begin{equation}
  \Div \gradU^\timeStep_{\level,\textup{RT}} \lambda = \nabla p^\timeStep_{\level,\textup{RT}} \lambda \in [\polynomials_{p-1}(\elem)]^3,
 \end{equation}
 which immediately allows to deduce that $\gradU^\timeStep_{\level,\textup{RT}} \in [\polynomials_{p}(\elem)]^{3 \times 3}$. We denote
 \begin{gather*}
  e^\timeStep_\level \gradU = \gradU^\timeStep_{\level,\textup{BDM}} \lambda - \gradU^\timeStep_{\level,\textup{RT}} \lambda, \qquad
  e^\timeStep_\level u = u ^\timeStep_{\level,\textup{BDM}} \lambda - \Pi_{\level,p-1} u^\timeStep_{\level,\textup{RT}} \lambda, \\
  e^\timeStep_\level p = p^\timeStep_{\level,\textup{BDM}} \lambda - p^\timeStep_{\level,\textup{RT}} \lambda,
 \end{gather*}
 and observe that $(e^\timeStep_\level \gradU, e^\timeStep_\level u, e^\timeStep_\level p) \in [\polynomials_p(\elem)]^{3 \times 3} \times [\polynomials_{p-1}(\elem)]^3 \times \polynomials_p(\elem)$ satisfies
 \begin{subequations}\begin{align}
  (e^\timeStep_\level \gradU, G)_\elem + (e^\timeStep_\level u, \Div G)_\elem & = 0 && \forall G \in [\polynomials_p(\elem)]^{3 \times 3} \\
  (-\Div e^\timeStep_\level \gradU + \nabla e^\timeStep_\level p, v)_\elem & = 0 && \forall v \in [\polynomials_{p-1}(\elem)]^3 \\
  \tfrac{1}{\timeStep} (e^\timeStep_\level p, q)_\elem - (e^\timeStep_\level u, \nabla q) & = 0 && \forall q \in \polynomials_p(\elem).
 \end{align}\end{subequations}
 The well-posedness of \eqref{EQ:local_sol} for BDM-H implies that $e^\timeStep_\level \gradU = 0$, $e^\timeStep_\level u = 0$, and $e^\timeStep_\level p = 0$, which implies the result.
\end{proof}

With these preliminaries done we can now state the main theorem of this section:
\begin{theorem}\label{TH:bilinear_form_identical}
 For all $\lambda, \mu \in \skeletalSpace_\level$, we have that
 \begin{equation}
  a^\timeStep_{\level,\textup{BDM}}(\lambda,\mu) = a^\timeStep_{\level,\textup{RT}}(\lambda,\mu) = a^\timeStep_{\level,\textup{SFH}}(\lambda,\mu).
 \end{equation}
\end{theorem}
\begin{proof}
 Lemmas \ref{LEM:sfh_bdm_loc} and \ref{LEM:rt_bdm_loc} say that for all $\lambda \in \skeletalSpace_\level$, we have
 \begin{equation}
  \gradU^\timeStep_{\level,\textup{BDM}} \lambda = \gradU^\timeStep_{\level,\textup{RT}} \lambda = \gradU^\timeStep_{\level,\textup{SFH}} \lambda
  \qquad \text{and} \qquad
  p^\timeStep_{\level,\textup{BDM}} \lambda = p^\timeStep_{\level,\textup{RT}} \lambda = p^\timeStep_{\level,\textup{SFH}} \lambda,
 \end{equation}
 which means that the first and third terms of \eqref{EQ:bilinear} are identical in all three cases. Note that $u^\timeStep_{\level,\textup{SFH}} \lambda - \lambda = 0$ on $\face^\star_\elem$, while $\tau_\level = 0$ on $\partial \elem \setminus \face^\star_\elem$ for SFH, and $\tau_\level = 0$ everywhere for BDM-H and RT-H. Thus, the second term in \eqref{EQ:bilinear} vanishes in all three cases.
\end{proof}

Similarly to the proof of Lemma \ref{LEM:sfh_bdm_loc}, we can deduce the following Lemma:
\begin{lemma}\label{LEM:sfh_bdm_loc_m}
 For all $\lambda \in \skeletalSpace_\level$, we have
 \begin{align*}
  \gradU^\timeStep_{\level,\textup{BDM}} m & = \gradU^\timeStep_{\level,\textup{SFH}} m, &
  p^\timeStep_{\level,\textup{BDM}} m & = p^\timeStep_{\level,\textup{SFH}} m, \\
  u^\timeStep_{\level,\textup{SFH}} m & = 0 \text{ on } \face^\star_\elem, &
  u^\timeStep_{\level,\textup{BDM}} m & = \Pi_{\level,p-1} u^\timeStep_{\level,\textup{SFH}} m.
 \end{align*}

\end{lemma}

In the following lemma, we investigate the relation between BDM-H and SFH with respect to the local problem \eqref{EQ:local_sol_f}.

\begin{lemma}\label{LEM:sfh_bdm_loc_f}
 The unknowns $\gradU^\timeStep_{\level,\textup{SFH}} f$, and $p^\timeStep_{\level,\textup{SFH}} f$ are independent of $\tau^\star_\level$, while $u^\timeStep_{\level,\textup{SFH}} f$ depends on $\tau^\star_\level$.
\end{lemma}
\begin{proof}
 Denote $V_{\elem,\textup{BDM}} = [\polynomials_{p-1}(\elem)]^3$, and
 \begin{equation}
  V^\bot_{\elem,\textup{BDM}} = \left\{ w \in [\polynomials_p(\elem)]^3 \colon (w, \xi) = 0, \; \forall \xi \in V_{\elem,\textup{BDM}} \right\}.
 \end{equation}
 By \cite[Lem.\ A.1 and A.2]{CockburnGS10}, there is an $\eta \in [\polynomials_p(\face^\star_\elem)]^3$ such that
 \begin{equation}
  \llangle \eta, v \rrangle_{\face^\star_\elem} = (f,v) \qquad \forall v \in V^\bot_{\elem,\textup{BDM}}. \label{EQ:edge_star}
 \end{equation}
 Suppose that $(\tilde \gradU^\timeStep_\level f, \tilde u^\timeStep_\level f, \tilde p^\timeStep_\level f) \in [\polynomials_p(\elem)]^{3 \times 3} \times [\polynomials_{p-1}(\elem)]^3 \times \polynomials_p(\elem)$ satisfy
 \begin{subequations}\label{EQ:tildef}\begin{align}
  (\tilde \gradU^\timeStep_\level f, G) + (\tilde u^\timeStep_\level f, \Div G) & = 0, \label{EQ:tildef_1} \\
  (-\Div \tilde \gradU^\timeStep_\level f + \nabla \tilde p^\timeStep_\level f, v) & = (f,v) - \llangle \eta, v \rrangle_{\face^\star_\elem}, 
  \label{EQ:tildef_2}\\
  \tfrac{1}{\timeStep} (\tilde p^\timeStep_\level f, q) - (\tilde u^\timeStep_\level f, \nabla q) & = 0 \label{EQ:tildef_3}
 \end{align}\end{subequations}
 for all $(G, v, q) \in [\polynomials_p(\elem)]^{3 \times 3} \times [\polynomials_{p-1}(\elem)]^3 \times \polynomials_p(\elem)$. We construct $\hat u^\timeStep_\level f\in [\polynomials_{p}(\elem)]^3$ by
 \begin{subequations}\label{EQ:cons_u}\begin{align}
  (\hat u^\timeStep_\level f, w) & = (\tilde u^\timeStep_\level f, w) && \forall w \in V_{\elem,\textup{BDM}}, \label{EQ:cons_u_1}\\
  \tau^\star_\level \hat u^\timeStep_\level f & = \eta && \text{on } \face^\star_\elem.\label{EQ:cons_u_F}
 \end{align}\end{subequations}
 Next, we show that $\tilde \gradU^\timeStep_\level f = \gradU^\timeStep_{\level,\textup{SFH}} f$, $\tilde p^\timeStep_\level f = p^\timeStep_{\level,\textup{SFH}} f$, and $\hat u^\timeStep_\level f = u^\timeStep_{\level,\textup{SFH}} f$. For all $G \in [\polynomials_p(\elem)]^{3 \times 3}$ using \eqref{EQ:tildef_1} and \eqref{EQ:cons_u_1}, we have
 \begin{equation}\label{EQ:tilde_hat}
  (\tilde \gradU^\timeStep_\level f, G) + (\hat u^\timeStep_\level f, \Div G) = 0,
 \end{equation}
 and we can decompose any $v \in [\polynomials_p(\elem)]^3$ as $v = v_1 + v_2$ with $v_1 \in V_{\elem,\textup{BDM}}$ and $v_2 \in V^\bot_{\elem,\textup{BDM}}$. Thus using \eqref{EQ:tildef_2}, \eqref{EQ:cons_u_F} and the fact that $- \Div \tilde \gradU^\timeStep_\level + \nabla \tilde p^\timeStep_\level f\in V_{\elem,\textup{BDM}}$, we have
 \begin{align}
  (- \Div \tilde \gradU^\timeStep_\level f + \nabla \tilde p^\timeStep_\level f, v) & = (- \Div \tilde \gradU^\timeStep_\level f + \nabla \tilde p^\timeStep_\level f, v_1) \\
  & = (f,v_1) - \llangle \tau^\star_\level \hat u^\timeStep_\level f, v_1 \rrangle_{\face^\star_\elem}.
 \end{align}
 This implies that
 \begin{align}
  (- \Div \tilde \gradU^\timeStep_\level f & + \nabla \tilde p^\timeStep_\level f, v) + \tau^\star_\level \llangle \hat u^\timeStep_\level f, v \rrangle_{\face^\star_\elem} \\
  =~ & (f, v_1)- \llangle \tau^\star_\level \hat u^\timeStep_\level f, v_1 \rrangle_{\face^\star_\elem} + \tau^\star_\level \llangle \hat u^\timeStep_\level f, v \rrangle_{\face^\star_\elem} \\
  =~ & (f, v_1) + \tau^\star_\level \llangle \hat u^\timeStep_\level f, v_2 \rrangle_{\face^\star_\elem} = (f,v),
 \end{align}
 where the last equality of the last line is a consequence of \eqref{EQ:cons_u_F} and \eqref{EQ:edge_star}. In return, this implies---combined with \eqref{EQ:cons_u_1}, \eqref{EQ:tilde_hat} and \eqref{EQ:tildef_3}---that $\tilde \gradU^\timeStep_\level f =  \gradU^\timeStep_{\level,\textup{SFH}} f$, $\tilde p^\timeStep_\level f = p^\timeStep_{\level, \textup{SFH}}$, and that $\hat u^\timeStep_\level f = u^\timeStep_{\level,\textup{SFH}} f$. Since $\tilde \gradU^\timeStep_\level f$ and $\tilde p^\timeStep_\level f$ are independent of $\tau^\star_\level$, while $\hat u^\timeStep_\level f$ depends on $\tau^\star_\level$, we receive the lemma.
\end{proof}

\begin{theorem}
 The unknowns $\lambda^n_{\level,\textup{SFH}}$, $p^n_{\level,\textup{SFH}}$, and $\gradU^n_{\level,\textup{SFH}}$ are independent of the choice of $\tau^\star_\level$, while $u^n_{\level,\textup{SFH}}$ depends on $\tau^\star_\level$.
\end{theorem}
\begin{proof}
 We know by Theorem \ref{TH:bilinear_form_identical} that bilinear form $a^\timeStep_{\level,\textup{SFH}}$ is independent of $\tau^\star_\level$. Furthermore, Lemma \ref{LEM:sfh_bdm_loc} says that $u^\timeStep_{\level,\textup{SFH}} \lambda$ can be determined as
 \begin{subequations}\begin{align}
  (u^\timeStep_{\level,\textup{SFH}} \lambda, v) & = (u^\timeStep_{\level,\textup{BDM}}, v) && \forall v \in [\polynomials_{p-1}(\elem)]^3, \\
  u^\timeStep_{\level,\textup{SFH}} \lambda & = \lambda && \text{on } \face^\star_\elem,
 \end{align}\end{subequations}
 which is independent of $\tau^\star_\level$. Thus, the right-hand side of \eqref{EQ:hdg_condensed} is independent of $\tau_\level$, which implies that $\lambda^n_{\level,\textup{SFH}}$ is independent of $\tau^\star_\level$.

 From \cite[Thm.\ 3.1]{NguyenPC10}, we know that
 \begin{subequations}\begin{align}
  \gradU^n_{\level,\textup{SFH}} & = \gradU^\timeStep_{\level,\textup{SFH}} \lambda^n_{\level,\textup{SFH}} + \gradU^\timeStep_{\level,\textup{SFH}} p^{n-1}_\level + \gradU^\timeStep_{\level,\textup{SFH}} f, \\
  u^n_{\level,\textup{SFH}} & = u^\timeStep_{\level,\textup{SFH}} \lambda^n_{\level,\textup{SFH}} + u^\timeStep_{\level,\textup{SFH}} p^{n-1}_\level + u^\timeStep_{\level,\textup{SFH}} f, \\
  p^n_{\level,\textup{SFH}} & = p^\timeStep_{\level,\textup{SFH}} \lambda^n_{\level,\textup{SFH}} + p^\timeStep_{\level,\textup{SFH}} p^{n-1}_\level + p^\timeStep_{\level,\textup{SFH}} f.
 \end{align}\end{subequations}
 Combining this with Lemmas \ref{LEM:sfh_bdm_loc}, \ref{LEM:sfh_bdm_loc_m} and \ref{LEM:sfh_bdm_loc_f} finishes the proof.
\end{proof}
\begin{remark}
 Numerical results in \cite{CockburnNP10} show that for the SFH method, $\| p - p^n_{\level,\textup{SFH}} \|$ and $\| \gradU - \gradU^n_{\level,\textup{SFH}} \|$ are independent of $\tau^\star_\level$, while $\| u - u^n_{\level,\textup{SFH}} \|$ varies for different choices of $\tau^\star_\level$. This validates the aforementioned theorem.
\end{remark}

The (condensed) stiffness matrices of BDM-H, RT-H, and SFH are identical, while the respective right-hand side vectors are different, see Theorem \ref{TH:bilinear_form_identical}. Hence, it is sufficient to prove the convergence of the V--cycle multigrid method for one of these schemes. To this end, we will focus on RT-H and will, for simplicity, omit the subscript `RT' in the following section. We define
\begin{subequations}\begin{align}
 A_\level\colon & \skeletalSpace_\level \to \skeletalSpace_\level \\
 & \langle A_\level \lambda, \mu \rangle_\level = a_\level^\timeStep(\lambda,\mu) \qquad \forall \mu \in \skeletalSpace_\level. 
\end{align}\end{subequations}
%
\section{Some properties of the local solvers}
%
In this section, we prove some properties of the RT-H local solvers which play an important role for the convergence analysis of multigrid method and give  the condition number for the HDG method. The following lemma uses the space of overall continuous, element-wise linear finite elements
\begin{equation}\label{EQ:lin_elem_def}
 \linElementSpace_\level = \{ v \in [C(\Omega)]^3\colon v_i|_\elem \text{ is linear } \forall \elem \in \mesh_\level, i = 1,\dots, 3 \}.
\end{equation}
Its proof is simple enough to be omitted.

\begin{lemma}\label{LEM:linear_constant}
 If $\mu = \gamma_\level w$ for some $w \in \linElementSpace_\level$, then
 \begin{equation}
  \gradU^\timeStep_\level \mu = \nabla w, \qquad u^\timeStep_\level \mu = w, \qquad p^\timeStep_\level \mu = - \timeStep \Div w.\label{EQ:linear}
 \end{equation}
 If $\mu \equiv c$ is constant on $\partial \elem$, then
 \begin{equation}
  \gradU^\timeStep_\level \mu = 0, \qquad u^\timeStep_\level \mu = c, \qquad p^\timeStep_\level \mu = 0.\label{EQ:constant}
 \end{equation}
\end{lemma}

\begin{lemma}\label{LEM:stability_local}
 For all $\mu \in \skeletalSpace_\level$, we have
 \begin{align}
  \| \gradU^\timeStep_\level \mu \|_0 & \lesssim \sqrt{1 + \timeStep} h^{-1}_\level \| \mu \|_\level, \label{EQ:stability_l}\\
  \| u^\timeStep_\level \mu \|_0 & \lesssim \sqrt{1 + \timeStep} \| \mu \|_\level, \label{EQ:stability_u}\\
  \| p^\timeStep_\level \mu \|_0 & \lesssim \sqrt{\timeStep (1 + \timeStep)} h^{-1}_\level \| \mu \|_\level \label{EQ:stability_p}.
 \end{align}
\end{lemma}
\begin{proof}
 Setting $\lambda = \mu$, $G = \gradU^\timeStep_\level \mu$, $v = u^\timeStep_\level \mu$, and $q = p^\timeStep_\level \mu$ in \eqref{EQ:local_sol}, we have
 \begin{multline}
  \| \gradU^\timeStep_\level \mu \|^2_{0,\elem} + \tfrac{1}{\timeStep} \| p^\timeStep_\level \mu \|^2_{0,\elem} = \llangle \mu, \gradU^\timeStep_\level \mu \Nu \rrangle_{\partial \elem} - \llangle \mu, p^\timeStep_\level \mu \Nu \rrangle_{\partial \elem} \\
  \le \tfrac{1}{2} \| \gradU^\timeStep_\level \mu \|^2_{0,\elem} + C h^{-1}_\level \nnorm \mu \nnorm^2_{\level, \partial \elem} + \tfrac{1}{2 (\timeStep)} \| p^\timeStep_\level \mu \|^2_{0,\elem} + C \timeStep h^{-1}_\level \nnorm \mu \nnorm^2_{\level, \partial \elem},
 \end{multline}
 where the inequality can be obtained using a combination of Young's inequality and the trace theorem. Simplifying this inequality and summing over all elements $\elem \in \mesh_\level$, we receive \eqref{EQ:stability_l} and \eqref{EQ:stability_p}.

 By \cite[Lem.\ 4.1]{CockburnG05}, we have
 \begin{align}
  \| u^\timeStep_\level \mu \|_0 & \lesssim h_\level \sup_{G \in W_\elem} \frac{(u^\timeStep_\level \mu, \Div G)}{\| G \|_{0,\elem}} \overset{\eqref{EQ:local_sol_gradu}}{=} h_\level \sup_{G \in W_\elem} \frac{\llangle \mu, G \Nu \rrangle_{\partial \elem} - (\gradU^\timeStep_\level \lambda, G)_\elem}{\| G \|_{0,\elem}} \notag\\
  & \lesssim h_\level \| \gradU^\timeStep_\level \lambda \|_{0,\elem} + \| \mu \|_{\level, \partial\elem},
 \end{align}
 which immediately implies \eqref{EQ:stability_u}.
\end{proof}

\begin{lemma}\label{LEM:stability_local_f}
 For all $f \in L^2(\Omega)$, we have
 \begin{align}
  \| \gradU^\timeStep_\level f \|_0 & \lesssim h_\level \| f \|_0, \label{EQ:stability_l_f}\\
  \| u^\timeStep_\level f \|_0 & \lesssim h^2_\level \| f \|_0, \label{EQ:stability_u_f}\\
  \| p^\timeStep_\level f  \|_0 & \lesssim \sqrt{\timeStep} h_\level \| f \|_0 \label{EQ:stability_p_f}.
 \end{align}
\end{lemma}
\begin{proof}
 Setting$G = \gradU^\timeStep_\level f$, $v = u^\timeStep_\level f$, and $q = p^\timeStep_\level f$ in \eqref{EQ:local_sol_f}, we have
 \begin{equation}
  \| \gradU^\timeStep_\level f \|^2_{0,\elem} + \tfrac{1}{\timeStep} \| p^\timeStep_\level f \|^2_{0,\elem} = (f, u^\timeStep_\level f)_\elem,\label{EQ:bound_f}
 \end{equation}
 and after applying H\"older's inequality we can bound
 \begin{align}
  \| u^\timeStep_\level f \|_{0,\elem} & \lesssim h_\level \sup_{G \in W_\elem} \frac{(u^\timeStep_\level f, \Div G)}{\| G \|_{0,\elem}} \notag\\
  & = h_\level \sup_{G \in W_\elem} \frac{(- \gradU^\timeStep_\level f, G)}{\| G \|_{0,\elem}} \le h_\level \| \gradU^\timeStep_\level f \|_{0,\elem},\label{EQ:bound_u}
 \end{align}
 where the first inequality again uses \cite[Lem.\ 4.1]{CockburnG05} and the equality is \eqref{EQ:local_sol_f_gradu}. Combining the above equations, we can deduce that
 \begin{equation}
  \| \gradU^\timeStep_\level f \|^2_{0,\elem} \lesssim \| f \|_{0,\elem} h_\level \| \gradU^\timeStep_\level f \|_{0,\elem}, 
 \end{equation}
 which implies \eqref{EQ:stability_l_f}. The remaining inequalities \eqref{EQ:stability_u_f} and \eqref{EQ:stability_p_f} are, again, direct consequences of \eqref{EQ:stability_l_f}, \eqref{EQ:bound_f}, and \eqref{EQ:bound_u}.
\end{proof}

We define the lifting
\begin{equation*}
 \liftingOp^\timeStep_\level \colon \skeletalSpace_\level \to \liftingOp^\timeStep_\level \skeletalSpace_\level \subset  \contElementSpace_{\level, p + 4} = \{ v \in C(\Omega) \mid v \in [\polynomials_{p+4}(\elem)]^3\quad \forall \elem \in \mesh_\level \},
\end{equation*} which can be interpreted as a vector version of $\liftingOp_\level$ in \cite{LuRK21}, via
\begin{subequations}\label{EQ:lifting_def}\begin{align}
 (\liftingOp^\timeStep_\level \lambda, v)_\elem & = (u^\timeStep_\level \lambda, v)_\elem && \forall v \in [\polynomials_p(\elem)]^3, \\
 \langle \liftingOp^\timeStep_\level \lambda, \eta \rangle_\face & = \langle \lambda, \eta \rangle_\face && \forall \eta \in [\polynomials_{p+1}(\face)]^3, \; \face \subset \partial \elem, \\
 \langle \liftingOp^\timeStep_\level \lambda, \xi \rangle_\edge & = \langle \avg{\lambda}_\edge, \xi \rangle_\edge && \forall \xi \in [\polynomials_{p+2}(\edge)]^3, \; \edge \text{ is edge of } \elem, \\
 [\liftingOp^\timeStep_\level \lambda](\vertex) & = \avg{\lambda}(\vertex) && \forall \vertex \text{ is vertex of } \elem.
\end{align}\end{subequations}
Here $\avg{\lambda}$ is the average taken over all cells adjacent to vertex $a$ or edge $\edge$.
Obviously, we have that
\begin{equation}
 \liftingOp^\timeStep_\level \gamma_\level w = w \qquad \text{ for } w \in \linElementSpace_\level,
\end{equation}
and by the standard scaling argument and \eqref{EQ:stability_u} we can deduce that
\begin{equation}\label{EQ:norm_equiv}
 \| \lambda \|_{\level} \lesssim \| \liftingOp^\timeStep_\level \lambda \|_{0} \lesssim \sqrt{1 + \timeStep} \| \lambda \|_{\level} \qquad \forall \lambda \in \skeletalSpace_\level.
\end{equation}

\begin{lemma}\label{LEM:trace_diff}
 For all $\lambda \in \skeletalSpace_\level $, we have
 \begin{equation*}
  \| u^\timeStep_\level \lambda - \lambda \|_\level \lesssim h_\level \sqrt{1 + \timeStep} \| \gradU^\timeStep_\level \lambda \|_0.
 \end{equation*}
\end{lemma}
\begin{proof}
We set 
 \begin{equation*}
  \lambda = \begin{pmatrix} \lambda_1 \\ \lambda_2 \\ \lambda_3 \end{pmatrix}, \qquad m_\elem(\lambda) = \begin{pmatrix} m_\elem(\lambda_1) \\ m_\elem(\lambda_2) \\ m_\elem(\lambda_3) \end{pmatrix}, \qquad m_\elem(\lambda_i) = \tfrac{1}{|\partial \elem|} \int_{\partial \elem} \lambda_i \ds.
 \end{equation*}
 Furthermore, let $F_\elem = B_\elem \hat x + b$ be the affine mapping of the reference element $\hat \elem$ to $\elem$. We denote by $\hat \lambda$ the vector consisting of $\hat \lambda_i = \lambda_i \circ F_\elem$ for $i = 1, 2, 3$.
 
 From \cite[Thm.\ 7.1]{DemkowiczGS09} we know that for each $\hat \lambda_i$, there is a $\hat g_i$ such that
 \begin{equation}\label{EQ:property_ref}
  \Div \hat g_i = 0 \qquad \text{ and } \qquad \hat g_i \cdot \hat \Nu = \hat \lambda_i - m_{\hat \elem} (\hat \lambda_i),
 \end{equation}
 and that $\hat g_i$ satisfies
 \begin{equation}\label{EQ:ref_bound}
  \| \hat g_i \|_{0,\hat \elem} \lesssim \nnorm \hat \lambda_i - m_{\hat \elem} (\hat \lambda_i) \nnorm_{\partial\hat \elem}.
 \end{equation}
 Next, we set $g_i = \tfrac{1}{\det(B_\elem)} B_\elem \hat g_i \circ F^{-1}_\elem$. Then, using \cite[Lem.\ 3.59 \& (3.81)]{Monk03} we have
 \begin{equation}
  \Div g_i = 0 \qquad \text{ and } \qquad \int_{\partial \elem} g_i \cdot \Nu \lambda_i \ds = \operatorname{sign}(
  \det(B_\elem)) \int_{\partial \hat \elem} \hat g_i \cdot \hat \Nu \hat \lambda_i \; \textup d\hat\sigma.
 \end{equation}
 Setting $G = (g_1, g_2, g_3)^\transposed$ in \eqref{EQ:local_sol_gradu}, we receive
 \begin{align}
  (\gradU^\timeStep_\level \lambda, G)_\elem & = \llangle \lambda, G \Nu \rrangle_\elem = \sum_{i=1}^3 \llangle \lambda_i, g_i \Nu \rrangle_{\partial \elem} \\
  & = \operatorname{sign}(\det(B_\elem)) \sum_{i=1}^3 \llangle \hat \lambda_i, \hat g_i \hat \Nu \rrangle_{\partial \hat \elem} \\
  & \overset{\eqref{EQ:property_ref}}{=} \operatorname{sign}(\det (B_\elem)) \llangle \hat \lambda_i - m_{\hat \elem}(\hat \lambda_i), \hat \lambda_i - m_{\hat \elem} (\hat \lambda_i) \rrangle_{\partial \hat \elem} \\
  & = \operatorname{sign}(\det (B_\elem)) \nnorm \hat \lambda - m_{\hat \elem} (\hat \lambda) \nnorm_{\partial \hat \elem}^2,
 \end{align}
 which means that
 \begin{equation}
  \nnorm \hat \lambda - m_{\hat \elem}(\hat \lambda) \nnorm^2_{\partial \hat \elem} \le \| \gradU^\timeStep_\level \lambda \|_{0,\elem} \| G \|_{0,\elem} \overset{\eqref{EQ:ref_bound}}{\underset{\text{scaling}}\lesssim} h^{-1/2}_\level \| \gradU^\timeStep_\level \lambda \|_{0,\elem} \nnorm  \hat \lambda - m_{\hat \elem} (\hat \lambda) \nnorm_{\partial \hat \elem}.
 \end{equation}
 Using this observation and that $m_\elem(\lambda_i)$ is the best approximation of $\lambda$ on $\partial \elem$ by a constant ($L^2$ order of convergence is one), we immediately have that
 \begin{equation}\label{EQ:approx_bound}
  \nnorm \lambda - m_\elem (\lambda) \nnorm_{\level,\partial \elem} \lesssim  h_\level \nnorm \hat \lambda - m_{\hat \elem}(\hat \lambda) \nnorm_{\partial \hat \elem} \lesssim h^{1/2}_\level \| \gradU^\timeStep_\level \lambda \|_{0,\elem}.
 \end{equation}
 Using \eqref{EQ:constant}, \eqref{EQ:stability_u} and the above inequality, we have that
 \begin{align}
  \| u^\timeStep_\level \lambda - \lambda \|_\level & = \| u^\timeStep_\level (\lambda- m_\elem (\lambda)) - (\lambda - m_\elem (\lambda))\|_\level\\
  & \lesssim \sqrt{1+\timeStep} \|\lambda - m_\elem (\lambda)\|_\level \\
  & \lesssim h_\level \sqrt{1+\timeStep} \| \gradU^\timeStep_\level \lambda \|_0.
 \end{align}
\end{proof}

\begin{lemma}\label{LEM:equiv_bilinear}
 For all $\lambda \in \skeletalSpace_\level$, we have
 \begin{equation*}
  \| \gradU^\timeStep_\level \lambda - \nabla u^\timeStep_\level \lambda \|_0 + \| \tfrac{1}{\timeStep} p^\timeStep_\level \lambda + \Div u^\timeStep_\level \lambda \|_0 \lesssim h^{-1}_\level \| u^\timeStep_\level \lambda - \lambda \|_\level.
 \end{equation*}
\end{lemma}
\begin{proof}
 Using \eqref{EQ:local_sol_gradu} and Green's formula, we have
 \begin{equation}
  (\gradU^\timeStep_\level \lambda - \nabla u^\timeStep_\level \lambda, G)_\elem = \llangle \lambda - u^\timeStep_\level \lambda, G \Nu \rrangle_{\partial \elem}.
 \end{equation}
 Choosing $G = \gradU^\timeStep_\level \lambda - \nabla u^\timeStep_\level \lambda$ and using the trace inequality, we get
 \begin{equation}
  \| \gradU^\timeStep_\level \lambda - \nabla u^\timeStep_\level \lambda \|^2_{0,\elem} \lesssim h^{-1/2}_\level \nnorm \lambda - u^\timeStep_\level \lambda \nnorm_{\level,\partial\elem} \| \gradU^\timeStep_\level \lambda - \nabla u^\timeStep_\level \lambda \|_{0,\elem},
 \end{equation}
 i.e.,
 \begin{equation}
  \| \gradU^\timeStep_\level \lambda - \nabla u^\timeStep_\level \lambda \|_{0,\elem} \lesssim h^{-1/2}_\level \nnorm \lambda - u^\timeStep_\level \lambda \nnorm_{\level,\partial\elem} \lesssim h^{-1}_\level \| \lambda - u^\timeStep_\level \lambda \|_{\level,\partial\elem}.
 \end{equation}

 Analogously, using \eqref{EQ:local_sol_bal} we have
 \begin{equation}
  ( \tfrac{1}{\timeStep} p^\timeStep_\level \lambda + \Div u^\timeStep_\level \lambda, q )_\elem = \llangle u^\timeStep_\level \lambda - \lambda, q \cdot \Nu \rrangle_{\partial \elem}.
 \end{equation}
 Setting $q = \tfrac{1}{\timeStep} p^\timeStep_\level \lambda + \Div u^\timeStep_\level \lambda$, we have
 \begin{equation}
  \| \tfrac{1}{\timeStep} p^\timeStep_\level \lambda + \Div u^\timeStep_\level \lambda \|_{0,\elem} \lesssim h^{-1}_\level \| \lambda - u^\timeStep_\level \lambda \|_{\level,\partial\elem},
 \end{equation}
 which concludes the prove after having been summed over all elements.
\end{proof}

Let us define the averaging operator
\begin{gather*}
  \avgOp_\level\colon\discElementSpace_\level \to \linElementSpace_\level,
\end{gather*}
with $\linElementSpace_\level$ as described in \eqref{EQ:lin_elem_def}. This operator first sets the values in all non-boundary vertices $\vec x \not \in \partial \Omega$ according to
\begin{equation*}
  \left[\avgOp_\level u\right] (\vec x)
  = \frac1{n_{\vec x}}\sum_{i=1}^{n_{\vec x}} u_{|\elem_i}(\vec x).
\end{equation*}
Here, $n_{\vec x}$ is the number of cells $\elem_i$ meeting in vertex $\vec x$ and $u_{|\elem_i}$ is the restriction of a function $u\in \discElementSpace_\level$ to cell $\elem_i$, which is single valued at $\vec x$. For $\vec x\in \partial\Omega$, we let $\left[\avgOp_\level u\right] (\vec x) = 0$. Afterwards, it uses the canonical interpolation operator $\linearInterpolation$ into linear finite elements to turn the `averaged' vertex values into a linear finite element function. Following the proofs of Lemmas 5.2 and 5.3 in \cite{LuRK21}, we can get the following corresponding two Lemmas.
\begin{lemma}\label{LEM:bound_by_local_l}
 For all $\lambda \in \skeletalSpace_\level$, we have
 \begin{align*}
  | \avgOp_\level u^\timeStep_\level \lambda |_1 & \lesssim \sqrt{1 + \timeStep} \| \gradU^\timeStep_\level \lambda \|_0, \\
  \| \avgOp_\level u^\timeStep_\level \lambda - u^\timeStep_\level \lambda \|_0 & \lesssim \sqrt{1 + \timeStep} h_\level \| \gradU^\timeStep_\level \lambda \|_0.
 \end{align*}
\end{lemma}
\begin{lemma}\label{LEM:diff_bound_by_l}
 For all $\lambda \in \skeletalSpace_\level$, we have
 \begin{equation*}
  \| \lambda - \gamma_\level \avgOp_\level u^\timeStep_\level \level \|_\level \lesssim h_\level \sqrt{1 + \timeStep} \| \gradU^\timeStep_\level \lambda \|_0.
 \end{equation*}
\end{lemma}

\begin{lemma}\label{LEM:norm_equiv_s}
 For all $\lambda \in \skeletalSpace_\level$, we have the following norm equivalence:
 \begin{equation*}
  \| \gradU^\timeStep_\level \lambda \|_{0} \lesssim \| \nabla \liftingOp^\timeStep_\level \lambda \|_{0} \lesssim \sqrt{1 + \timeStep} \| \gradU^\timeStep_\level \lambda \|_{0}.
 \end{equation*}
\end{lemma}
\begin{proof}
 From the definition of $\liftingOp^\timeStep_\level \lambda$ and \eqref{EQ:local_sol_gradu}, we have
 \begin{equation}
  (\gradU^\timeStep_\level \lambda, G)_\elem = - (\liftingOp^\timeStep_\level \lambda, \Div G)_\elem + \llangle \liftingOp^\timeStep_\level \lambda, G \Nu \rrangle_{\partial \elem} = (\nabla \liftingOp^\timeStep_\level \lambda, G)_\elem
 \end{equation}
 for all $G \in [\polynomials_p(\elem)]^{3 \times 3}$. This implies that $\gradU^\timeStep_\level \lambda = \Pi_\level \nabla \liftingOp^\timeStep_\level \lambda$, where $\Pi_\level$ is the $L^2$ projection to $[\polynomials_p(\elem)]^{3 \times 3}$, which gives the first inequality. With the above Lemmas, similar to the proof of \cite[(5.14)]{LuRK21}, we can prove the second inequality.
\end{proof}

\begin{lemma}\label{LEM:eigenvalue_bound}
 For all $\lambda \in \skeletalSpace_\level$, we have
 \begin{equation*}
  \| \lambda \|^2_\level \lesssim a^\timeStep_\level (\lambda, \lambda) \lesssim (1 + \timeStep) h^{-2}_\level \| \lambda \|^2_\level.
 \end{equation*}
\end{lemma}
\begin{proof}
 The upper bound is a combination of \eqref{EQ:bilinear}, \eqref{EQ:stability_l}, and \eqref{EQ:stability_p}. For the lower bound, we observe that
 \begin{equation}
  \| \lambda \|^2_\level \lesssim \sum_{\elem \in \mesh_\level} \tfrac{1}{h_\level} \nnorm \lambda - m_\elem(\lambda) \nnorm^2_{\level,\partial\elem} \lesssim \| \gradU^\timeStep_\level \lambda \|^2_0 \lesssim a^\timeStep_\level (\lambda, \lambda),
 \end{equation}
 where the first inequality is \cite[(2.12)]{Gopalakrishnan03} and the second inequality is \eqref{EQ:approx_bound}.
\end{proof}

\begin{remark}
 Consequently, the spectral condition number of the stiffness matrix is $\mathcal O((1 + \timeStep) h^{-2}_\level)$, which is validated by the observations in the numerical results in \cite{NguyenPC10}.
\end{remark}

\begin{remark}\label{REM:ls}
 The properties that have been presented in this section correspond to (LS1) -- (LS6) in \cite{LuRK21} for the Stokes equation. Specifically, Lemma \ref{LEM:trace_diff} is (LS1), Lemma \ref{LEM:stability_local} is (LS2), Lemma \ref{LEM:equiv_bilinear} is (LS3), Lemma \ref{LEM:linear_constant} is (LS4), Lemma \ref{LEM:stability_local_f} and Theorem \ref{TH:error_auxiliary} play the role of (LS5), and Lemma \ref{LEM:eigenvalue_bound} is (LS6).
\end{remark}
%
\section{Error estimate of the auxiliary problem}
%
We consider the auxiliary problem
\begin{subequations}\label{EQ:ap}\begin{align}
 \gradU - \nabla u & = 0 && \text{ in } \Omega, \\
 -\Div \gradU + \nabla p & = f && \text{ in } \Omega, \\
 \tfrac{1}{\timeStep} p + \Div u & = 0 && \text{ in } \Omega, \\
 u & = 0 && \text{ on } \partial \Omega,
\end{align}\end{subequations}
for which we assume that the regularity assumption
\begin{subequations}\label{EQ:regularity_est}\begin{align}
 | \gradU |_{1,\Omega} + | u |_{2,\Omega} & \le C_\timeStep \| f \|_{0,\Omega}, \\
 |p|_{1,\Omega} = \timeStep | \Div u |_{1,\Omega} & \le C_\timeStep \timeStep \| f \|_{0,\Omega}
\end{align}\end{subequations}
holds for some constant $C_\timeStep$ independent of $f$.

We denote the RT-H approximation of \eqref{EQ:ap} by $(\gradU^\timeStep_\level, u^\timeStep_\level, p^\timeStep_\level, \lambda^\timeStep_\level) \in W_\level \times V_\level \times Q_\level \times \skeletalSpace_\level$. That is, we assume that
\begin{subequations}\label{EQ:rth_ap}\begin{align}
 (\gradU^\timeStep_\level, G) + (u^\timeStep_\level, \Div G) - \llangle \lambda^\timeStep_\level, G \Nu \rrangle_{\level} & = 0, \\
 (- \Div \gradU^\timeStep_\level + \nabla p^\timeStep_\level, v) & = (f,v), \\
 \tfrac{1}{\timeStep} (p^\timeStep_\level, q) - (u^\timeStep_\level, q) + \llangle \lambda^\timeStep_\level, q \Nu \rrangle_{\level} & = 0, \\
 \llangle - \gradU^\timeStep_\level \Nu + p^\timeStep_\level \Nu, \mu \rrangle_{\level} & = 0
\end{align}\end{subequations}
for all $(G, v, q, \mu) \in W_\level \times V_\level \times Q_\level \times \skeletalSpace_\level$.

We will estimate the error of the approximate solution of \eqref{EQ:rth_ap} under the regularity assumption \eqref{EQ:regularity_est}.
\begin{theorem}\label{TH:error_auxiliary}
 We have the error bound
 \begin{equation*}
  \sqrt{\tfrac{1}{\timeStep}} \| p - p^\timeStep_\level \|_0 + \| \gradU - \gradU^\timeStep_\level \|_0 \lesssim C_\timeStep (1 + \timeStep) h_\level \| f \|_0.
 \end{equation*}
\end{theorem}
\begin{proof}
 We denote the Raviart--Thomas projection to $W_\level$ by $\Pi^\textup{RT}_\level$. It is supposed to satisfy
 \begin{subequations}\begin{align}
  ( \Pi^\textup{RT}_\level \gradU_i , v) & = (\gradU_i, v) && \forall v \in [\polynomials_{p-1}(\elem)]^3, \\
  \langle \Pi^\textup{RT}_\level \gradU_i \cdot \Nu, \mu \rangle_\face & = \langle \gradU_i \cdot \Nu, \mu \rangle_\face && \forall \face \subset \partial \elem, \; \forall \mu \in \polynomials_p(\face)
 \end{align}\end{subequations}
 for $i = 1, 2, 3$. Here, 
 \begin{equation}
  \gradU = \begin{pmatrix} \gradU_1 \\ \gradU_2 \\ \gradU_3 \end{pmatrix} \qquad \text{ and } \qquad \Pi^\textup{RT}_\level \gradU = \begin{pmatrix} \Pi^\textup{RT}_\level \gradU_1 \\ \Pi^\textup{RT}_\level \gradU_2 \\ \Pi^\textup{RT}_\level \gradU_3 \end{pmatrix}.
 \end{equation}
 Operator $\Pi_\level$ denotes the $L^2$ projection to $Q_\level$ or $M_\level$, i.e.,
 \begin{align}
  (\Pi_\level p, v)_\elem & = (p, v)_\elem && \forall v \in \polynomials_p(\elem),\\
  (\Pi_\level u, v)_\elem & = (u, v)_\elem && \forall v \in [\polynomials_p(\elem)]^3,
 \end{align}
 and $\Pi^\partial_\level$ is the $L^2$ projection to $\skeletalSpace_\level$, i.e.,
 \begin{align}
  \llangle \Pi^\partial_\level u, \eta \rrangle_{\skeleton_\level} & = \llangle u, \eta \rrangle_{\skeleton_\level} && \forall \eta \in \skeletalSpace_\level
 \end{align}
 In this manuscript, we denote $\Pi_\level$ as the $L^2$ projections to $[\polynomials_p(\elem)]^{3 \times 3}$, $[\polynomials_p(\elem)]^3$ and $\polynomials_p(\elem)$. We made this slight abuse of notation to avoid introducing additional symbols. So does the  $L^2$ projections to the skeletons $\Pi^\partial_\level$.
 With these definitions, we have
 \begin{subequations}\label{EQ:proj_ap}\begin{align}
  (\Pi^\textup{RT}_\level \gradU, G) + (\Pi_\level u, \Div G) - \llangle \Pi^\partial_\level u, G \Nu \rrangle_{{\level}} & = (\Pi^\textup{RT}_\level \gradU - \gradU, G), \\
  (-\Div \Pi^\textup{RT}_\level \gradU + \nabla \Pi_\level p, v) + \llangle \Pi^\partial_\level p - \Pi_\level p, v \Nu \rrangle_{{\level}} & = (f,v), \\
  \tfrac{1}{\timeStep} (\Pi_\level p, q) - (\Pi_\level u, \nabla q) + \llangle \Pi^\partial_\level u, q \Nu \rrangle_{{\level}} & = 0, \\
  \llangle - \Pi^\textup{RT}_\level \gradU \Nu + \Pi^\partial_\level p \Nu, \mu \rrangle_{{\level}} & = 0
 \end{align}\end{subequations}
 for all $(G, v, q, \mu) \in W_\level \times V_\level \times Q_\level \times \skeletalSpace_\level$.

 We denote $e_\gradU = \Pi^\textup{RT}_\level \gradU - \gradU^\timeStep_\level$, $e_u = \Pi_\level u - u^\timeStep_\level$, $e_\lambda = \Pi^\partial_\level u - \lambda^\timeStep_\level$, $e_p = \Pi_\level p - p^\timeStep_\level$. Thus, \eqref{EQ:rth_ap} and \eqref{EQ:proj_ap} result in

 \begin{subequations}\label{EQ:diff_ap}\begin{align}
  (e_\gradU, G) + (e_u, \Div G) - \llangle e_\lambda, G \Nu \rrangle_{{\level}} & = (\Pi^\textup{RT}_\level \gradU - \gradU, G), \\
  (- \Div e_\gradU + \nabla e_p, v) + \llangle \Pi^\partial_\level p - \Pi_\level p, v \cdot \Nu \rrangle_{{\level}} & = 0, \\
  \tfrac{1}{\timeStep} (e_p, q) - (e_u, \nabla q) + \llangle e_\lambda, q \Nu \rrangle_{{\level}} & = 0, \\
  \llangle -e_\gradU \Nu + e_p \Nu + (\Pi^\partial_\level p - \Pi_\level p) \Nu, \mu \rrangle_{{\level}} & = 0
 \end{align}\end{subequations}
 for all $(G, v, q, \mu) \in W_\level \times V_\level \times Q_\level \times \skeletalSpace_\level$. Setting $G = e_\gradU$, $v = e_u$, $q = e_p$, $\mu = e_\lambda$, we receive
 \begin{equation}\label{EQ:error_a}
  \| e_\gradU \|^2_{0,\Omega} + \tfrac{1}{\timeStep} \| e_p \|^2_{0,\Omega} = ( \Pi^\textup{RT}_\level \gradU - \gradU, e_\gradU ) - \llangle (\Pi^\partial_\level p - \Pi_\level p) \Nu, e_u - e_\lambda \rrangle_{\level}.
 \end{equation}
 Using Green's formula, we receive
 \begin{equation}
  (\Pi^\textup{RT}_\level \gradU - \gradU, G) = (e_\gradU - \nabla e_u, G) + \llangle e_u - e_\lambda, G \Nu \rrangle_{\level}
 \end{equation}
 on each $\elem \in \mesh_\level$. There is $G \in W_\elem$ such that
 \begin{subequations}\begin{align}
  G \Nu & = e_u - e_\lambda && \text{ on } \partial \elem, \\
  (G, v) & = 0 && \text{ for all } v \in [\polynomials_{p-1}(\elem)]^{3 \times 3}, \\
  \| G \|_{0,\Omega} & \lesssim h^{1/2}_\level \nnorm e_u - e_\lambda \nnorm_\level.
 \end{align}\end{subequations}
 Hence, we have
 \begin{equation}
  \nnorm e_u - e_\lambda \nnorm^2_\level = (\Pi^\textup{RT}_\level \gradU - \gradU - e_\gradU, G) \lesssim h^{1/2}_\level  ( \| \Pi^\textup{RT}_\level \gradU - \gradU \|_0 + \| e_\gradU \|_0 ) \nnorm e_u - e_\lambda \nnorm_\level,
 \end{equation}
 which simplifies to $\nnorm e_u - e_\lambda \nnorm_\level \lesssim h^{1/2}_\level (\| \Pi^\textup{RT}_\level \gradU - \gradU \|_0 + \| e_\gradU \|_0 )$. Combining this and \eqref{EQ:error_a}, we receive
 \begin{multline}
  \| e_\gradU \|^2_{0,\Omega} + \tfrac{1}{\timeStep} \| e_p \|^2_{0,\Omega} \lesssim \| \Pi^\textup{RT}_\level \gradU - \gradU \|_{0,\Omega} \| e_\gradU \|_{0,\Omega} \\
  + \nnorm \Pi^\partial_\level p - \Pi_\level p \nnorm_\level h^{1/2}_\level (\| \Pi^\textup{RT}_\level \gradU - \gradU \|_0 + \| e_\gradU \|_0 ).
 \end{multline}
 Using Young's inequality and a simple algebraic manipulation, we obtain
 \begin{align}
  \| e_\gradU \|^2_0 + \tfrac{1}{\timeStep} \| e_p \|^2_0 \lesssim & \| \Pi^\textup{RT}_\level \gradU - \gradU \|^2_{0,\Omega} + h_\level \nnorm \Pi^\partial_\level p - \Pi_\level p \nnorm_\level^2 \\
  & + h^{1/2}_\level \nnorm \Pi^\partial_\level p - \Pi_\level p \nnorm_\level \| \Pi^\textup{RT}_\level \gradU - \gradU \|_0 \notag\\
  \lesssim & h^2_\level | \gradU |^2_{1, \Omega} + h^2_\level | p |^2_{1, \Omega} + h^2_\level | p |_{1, \Omega} | \gradU |_{1, \Omega} \\
  \lesssim & (1 + \timeStep)^2 C^2_\timeStep h^2_\level \| f \|^2_{0,\Omega},
 \end{align}
 where the second inequality follows from the approximation properties of the projections and the third inequality is \eqref{EQ:regularity_est}. Combining this with the properties of $\Pi^\textup{RT}_\level$ and $\Pi_\level$ finishes the proof.
\end{proof}
%
\section{Injection operators}\label{SEC:injection_op}
%
Assumptions on injection operators:
\begin{enumerate}
  \item Stability:
  \begin{equation}\label{EQ:ia1}
   \| \injectionOp_\level \lambda \|_\level \lesssim \| \lambda \|_{\level - 1} \qquad \forall \lambda \in \skeletalSpace_{\level-1} \tag{IA1}
  \end{equation}
  \item Identity for conforming finite elements:
  \begin{equation}\label{EQ:ia2}
   \injectionOp_\level \gamma_{\level - 1} w = \gamma_\level w \qquad \forall w \in \linElementSpace_{\level - 1} \tag{IA2}
  \end{equation}
\end{enumerate}

Possible injection operators are three dimensional versions of the four injection operators of \cite{LuRK21}, since they have already been shown to satisfy \eqref{EQ:ia1} and \eqref{EQ:ia2} for the Poisson equation.

Although all four operators can be used for our scheme, we restrict ourselves to the analysis of their interpolation operator $\injectionOp_\level^1$. It linearly interpolates the (averaged) values of a skeleton function in the corners of a new face. For a detailed description of this procedure, cf.\ \cite[Sect.\ 3.2]{LuRK21}.

\begin{lemma}
 Injection operator $\injectionOp_\level^1$ admits \eqref{EQ:ia1} and \eqref{EQ:ia2}.
\end{lemma}
\begin{proof}
 This is the proof for $\injectionOp_\level^1$ in \cite[Lem.\ 3.2]{LuRK21}.
\end{proof}
\begin{remark}[Other injection operators]
 Following the proof of \eqref{EQ:ia1} for the other injection operators in \cite{LuRK21} we note that \eqref{EQ:stability_u} involves $\timeStep$ (while the analogous equation for Poisson does not). This indicates that for the other injection operators of \cite{LuRK21} only
 $$ \| \injectionOp_\level^\star \lambda \|_\level \lesssim \sqrt{1+\timeStep} \| \lambda \|_{\level - 1}.$$
 holds true. This statement is sufficient for the remaining analysis to be conducted (with the factor popping up at different locations).
\end{remark}

\begin{lemma}\label{LEM:quasi_orthogonality}
 Assume \eqref{EQ:ia2}, then for any $\lambda \in \skeletalSpace_\level$, we have
 \begin{equation}
  ( \gradU^\timeStep_\level \lambda - \gradU^\timeStep_{\level - 1} \projectionOp_{\level - 1} \lambda, \nabla w ) - (p^\timeStep_\level \lambda - p^\timeStep_{\level -1} \projectionOp_{\level -1} \lambda, \Div w) = 0 \qquad \forall w \in \linElementSpace_{\level - 1},
 \end{equation}
 where $\projectionOp_{\level - 1}\colon \skeletalSpace_\level \to \skeletalSpace_{\level - 1}$ is defined via
 \begin{equation}
  a_{\level - 1}(\projectionOp_{\level-1} \lambda, \mu) = a_\level(\lambda, \injectionOp_\level \mu).
 \end{equation}
\end{lemma}
\begin{proof}
 For $w \in \linElementSpace_{\level - 1}$ let $\mu = \gamma_{\level - 1} w$. Using \eqref{EQ:ia2} and Lemma \ref{LEM:linear_constant} we have
 \begin{equation}
  \injectionOp_\level \mu = \gamma_\level w, \qquad \gradU^\timeStep_{\level - 1} \mu = \gradU^\timeStep_\level \injectionOp_\level \mu = \nabla w, \qquad p^\timeStep_{\level - 1} \mu = p^\timeStep_\level \injectionOp_\level \mu = - \timeStep \Div w,
 \end{equation}
 which gives the result.
\end{proof}
%
\section{Multigrid method and main convergence result}
%
We analyze the convergence of a standard V--cycle multigrid method based on Theorem 3.1 of~\cite{DuanGTZ07}. We assume that the smoother falls into their framework, define $\lambda^A_\level$ to be the largest eigenvalue of $A_\level$, and need to satisfy three assumptions. That is, there needs to be constants $C_1, C_2, C_3 > 0$ independent of the mesh level $\level$, such that we have:
\begin{itemize}
\item Regularity approximation assumption:
  \begin{equation}\label{EQ:a1}
    | a_\level(\lambda - \injectionOp_\level \projectionOp_{\level-1} \lambda, \lambda) |
    \le C_1 \frac{\| A_\level \lambda \|^2_\level}{\underline \lambda^A_\level} \qquad \forall \lambda \in \skeletalSpace_\level. \tag{A1}
  \end{equation}
\item Stability of the ``Ritz quasi-projection'' $\projectionOp_{\level-1}$ and injection $\injectionOp_\level:$
 \begin{equation}\label{EQ:a2}
  \| \lambda - \injectionOp_\level \projectionOp_{\level-1} \lambda\|_{a_\level} \le C_2 \| \lambda \|_{a_\level} \qquad \forall \lambda \in \skeletalSpace_\level. \tag{A2}
\end{equation}
\item Smoothing hypothesis:
 \begin{equation}\label{EQ:a3}
  \frac{\| \lambda \|^2_\level}{\underline \lambda^A_\level} \le C_3 \langle K_\level \lambda, \lambda \rangle_\level, \tag{A3}
 \end{equation}
 where $K_\level$ only depends on the selected smoother.
\end{itemize}
Having formulated these preliminaries, Theorem~3.1 in~\cite{DuanGTZ07} reads
\begin{theorem}\label{TH:main_theorem}
 Assume that \eqref{EQ:a1}, \eqref{EQ:a2}, and \eqref{EQ:a3} hold. Then for all $\level \ge 0$,
 \begin{equation}
  | a_\level ( \lambda - B_\level A_\level \lambda, \lambda ) | \le \delta a_\level(\lambda, \lambda),
 \end{equation}
 where
 \begin{equation}
  \delta = \frac{C_1 C_3}{\iterMgInner - C_1 C_3} \qquad \text{with} \qquad \iterMgInner > 2 C_1 C_3.
 \end{equation}
\end{theorem}

That is, we will show that the aforementioned three assumptions hold  in order to guarantee convergence of the multigrid method, where \eqref{EQ:a3} is obviously true for the considered standard smoothers: Jacobi and Gauss--Seidel.
%
\section{Convergence analysis}
%
\subsection{Poof of \eqref{EQ:a2}}
%
Since we have Stokes versions of (LS1) -- (LS6) in \cite{LuRK21}, following the proof of Lemma 5.1 in \cite{LuRK21}, we obtain:
\begin{lemma}\label{LEM:conv_result}
 Assuming \eqref{EQ:ia1} and \eqref{EQ:ia2}, we have for all $\lambda \in \skeletalSpace_{\level - 1}$ that
 \begin{align}
  \| \gradU^\timeStep_\level \injectionOp_\level \lambda \|_0 & \lesssim (1 + \timeStep) \| \gradU^\timeStep_{\level-1} \lambda \|_0, \label{EQ:conv_l}\\
 \| u^\timeStep_{\level - 1} \lambda - u^\timeStep_\level \injectionOp_\level \lambda \|_0 & \lesssim h_\level (1 + \timeStep) \| \gradU^\timeStep_{\level - 1} \lambda \|_0, \label{EQ:conv_u}\\
 \| p^\timeStep_{\level - 1} \lambda - p^\timeStep_\level \injectionOp_\level \lambda \|_0 & \lesssim (1 + \timeStep) \sqrt{\timeStep} \| \gradU^\timeStep_{\level - 1} \lambda \|_0. \label{EQ:conv_p}
 \end{align}
\end{lemma}
\begin{proof}
 We only give the proof of \eqref{EQ:conv_p}:
 \begin{equation}
  \| p^\timeStep_\level \injectionOp_\level \lambda - p^\timeStep_{\level - 1} \lambda \|_0 \le \underbrace{ \| p^\timeStep_\level \injectionOp_\level \lambda - p^\timeStep_\level \injectionOp_\level \gamma_{\level - 1} w \|_0 }_{ =: \Xi_1 } + \underbrace{ \| p^\timeStep_\level \injectionOp_\level \gamma_{\level - 1} w  - p^\timeStep_{\level - 1} \lambda \|_0 }_{ =: \Xi_2 }
 \end{equation}
 for all $w \in \linElementSpace_\level$.
 \begin{align}
  \Xi_1 & \overset{\eqref{EQ:stability_p}}{\lesssim} \sqrt{\timeStep (1 + \timeStep)} h^{-1}_\level \| \injectionOp_\level \lambda - \injectionOp_\level \gamma_{\level - 1} w \|_\level \\
  & \overset{\eqref{EQ:ia1}}{\lesssim} \sqrt{\timeStep (1 + \timeStep)} h^{-1}_\level \| \lambda - \gamma_{\level - 1} w \|_{\level -1},
 \end{align}
 \begin{align}
  \Xi_2 & \overset{\eqref{EQ:ia2}}{=} \| p^\timeStep_\level \gamma_\level w - p^\timeStep_{\level - 1} \lambda \|_0 \overset{\eqref{EQ:linear}}{=} \| p^\timeStep_{\level - 1} \gamma_{\level - 1} w - p^\timeStep_{\level - 1} \lambda \|_0 \\
  & \overset{\eqref{EQ:stability_p}}{\lesssim} \sqrt{\timeStep (1 +  \timeStep)} h^{-1}_\level \| \lambda - \gamma_{\level - 1} w \|_{\level - 1}.
 \end{align}
 Setting $w = \avgOp_{\level -1} u^\timeStep_{\level - 1} \lambda$ and using Lemma \ref{LEM:diff_bound_by_l}, we receive \eqref{EQ:conv_p}.
\end{proof}
\begin{lemma}
 We have that
 \begin{align}
  a_\level(\injectionOp_\level \lambda, \injectionOp_\level \lambda) & \lesssim (1 + \timeStep)^2 a_{\level - 1}(\lambda, \lambda) && \forall \lambda \in \skeletalSpace_{\level - 1}, \label{EQ:l_satble}\\
  a_{\level - 1} (\projectionOp_{\level - 1} \lambda, \projectionOp_{\level - 1} \lambda) & \lesssim (1 + \timeStep)^2 a_\level(\lambda, \lambda) && \forall \lambda \in \skeletalSpace_\level. \label{EQ:ap_stable}
 \end{align}
\end{lemma}
\begin{proof}
 We have
 \begin{align}
  a_\level(\injectionOp_\level \lambda, \injectionOp_\level \lambda) & = \| \gradU^\timeStep_\level \injectionOp_\level \lambda \|^2_0 + \tfrac{1}{\timeStep} \| p^\timeStep_\level \injectionOp_\level \lambda \|^2 \\
  & \lesssim (1 + \timeStep)^2 \| \gradU^\timeStep_{\level - 1} \lambda \|^2_0 + \tfrac{1}{\timeStep} \| p^\timeStep_{\level - 1} \lambda \|^2_0 \\
  & \le (1 + \timeStep)^2 a_{\level - 1}(\lambda, \lambda),
 \end{align}
 where the first inequality is \eqref{EQ:conv_l} and \eqref{EQ:conv_p}, respectively. Beyond this,
 \begin{align}
  \| \projectionOp_{\level - 1} \lambda \|^2_{a_{\level - 1}} & = a_\level(\lambda, \injectionOp_\level \projectionOp_{\level - 1}\lambda) \le \| \lambda \|_{a_\level} \| \injectionOp_\level \projectionOp_{\level - 1} \lambda \|_{a_\level} \\
  & \lesssim \| \lambda \|_{a_\level} (1 + \timeStep) \| \projectionOp_{\level - 1} \lambda \|{a_{\level - 1}}. 
 \end{align}
\end{proof}
%
\begin{theorem}
 \eqref{EQ:a2} holds true.
\end{theorem}
%
\begin{proof}
\begin{align}
 a_\level (\lambda & - \injectionOp_\level \projectionOp_{\level - 1} \lambda, \lambda - \injectionOp_\level \projectionOp_{\level - 1} \lambda) \\
 = & a_\level(\lambda, \lambda) \underbrace{ - 2 a_\level(\projectionOp_{\level - 1} \lambda, \projectionOp_{\level - 1} \lambda) }_{ \le 0 } + a_\level(\injectionOp_\level \projectionOp_{\level - 1} \lambda, \injectionOp_\level \projectionOp_{\level - 1} \lambda) \\
 \lesssim & a_\level(\lambda, \lambda) + (1 + \timeStep)^4 a_\level(\lambda, \lambda).
\end{align}
\end{proof}
%
\subsection{Proof of \eqref{EQ:a1}}
%
Here, we use the previously introduced $\liftingOp^\timeStep_\level$, see \eqref{EQ:lifting_def}, to construct the right-hand side of the auxiliary problem as an $L^2$ lifting. Hence, we introduce $f_\lambda$ by
\begin{equation}
 (f_\lambda, \liftingOp^\timeStep_\level \mu) = \langle A_\level \lambda , \mu \rangle_\level = a_\level(\lambda, \mu) \qquad \forall \mu \in \skeletalSpace_\level, \label{EQ:scp_def}
\end{equation}
where we search for $f_\lambda$ in the space $\liftingOp^\timeStep_\level \skeletalSpace_\level$. We denote $(\tilde \gradU, \tilde u, \tilde p)$ as the weak solution of
\begin{subequations}\label{EQ:ap_flambda}\begin{align}
 \tilde \gradU - \nabla \tilde u & = 0 && \text{ in } \Omega, \label{EQ:ap_flambda_1}\\
 - \Div \tilde \gradU + \nabla \tilde p & = f_\lambda && \text{ in } \Omega, \\
 \tfrac{1}{\timeStep} \tilde p + \Div \tilde u & = 0 && \text{ in } \Omega, \label{EQ:ap_flambda_3}\\
 \tilde u & = 0 && \text{ on } \partial \Omega.
\end{align}\end{subequations}
Let $\tilde \lambda \in \skeletalSpace_\level$ be the RT-H solution of \eqref{EQ:ap_flambda}, that is
\begin{equation}
 a_\level(\tilde \lambda, \mu) = (f_\lambda, u^\timeStep_\level \mu) \qquad \forall \mu \in \skeletalSpace_\level. \label{EQ:tilde_lambda}
\end{equation}

Next, we present an approximation result for $\tilde \lambda$. Its proof is similar to the one of Lemma 5.6 in \cite{LuRK21}, since $\liftingOp^\timeStep_\level$ has the same properties as $S_\level$ in \cite{LuRK21}.
\begin{lemma}\label{LEM:bound_flambda}
 We have that
 \begin{equation}
  \| \lambda - \tilde \lambda \|_{a_\level} \lesssim \sqrt{1 + \timeStep} h_\level \| A_\level \lambda \|_\level \qquad \text{ and } \qquad \| f _\lambda \|_0 \lesssim \| A_\level \lambda \|_\level. \label{EQ:bound_flambda}
 \end{equation}
\end{lemma}
\begin{lemma}\label{EQ:reconstruction_approximation}
 If the model problem has full elliptic regularity, and \eqref{EQ:ia1} and \eqref{EQ:ia2} hold, then for all $\lambda \in \skeletalSpace_\level$, there is a $\bar w \in \linElementSpace_{\level - 1}$ such that
 \begin{multline*}
  \| \gradU^\timeStep_\level \lambda - \nabla \bar w \|_0 + \tfrac{1}{\sqrt{\timeStep}} \| p^\timeStep_\level \lambda + \timeStep \Div \bar w \|_0 + \| \gradU^\timeStep_{\level - 1} \projectionOp_{\level - 1} \lambda - \nabla \bar w \|_0 \\
  + \tfrac{1}{\sqrt{\timeStep}} \| p^\timeStep_{\level - 1} \projectionOp_{\level - 1} \lambda + \timeStep \Div \bar w \|_0 \le \bar C h_\level \| A_\level \lambda \|_\level,
 \end{multline*}
 where $\bar C = C_1 ( C_\timeStep (1 + \timeStep) + C_2 (1 + \timeStep)^{3/2} )$, in which $C_1$ and $C_2$ are independent of the mesh size and $\timeStep$.
\end{lemma}
\begin{proof}
 We only prove the inequality with respect to the third and forth terms, since the first and second terms can be treated analogously. Setting $e_{\level - 1} = \projectionOp_{\level - 1} \lambda - \tilde \lambda_{\level - 1}$, we have
 \begin{align}
  \| \gradU^\timeStep_{\level - 1} & \projectionOp_{\level - 1} \lambda - \nabla \bar w \|_0 + \tfrac{1}{\sqrt{\timeStep}} \| p^\timeStep_{\level - 1} \projectionOp_{\level - 1} \lambda + \timeStep \Div \bar w \|_0 \\
  \le & \| \gradU^\timeStep_{\level - 1} e_{\level - 1} \|_0 + \| \gradU^\timeStep_{\level - 1} \tilde \lambda_{\level - 1} - \nabla \bar w \|_0 \\
  & + \tfrac{1}{\sqrt{\timeStep}} \| p^\timeStep_{\level - 1} e_{\level - 1} \|_0 + \tfrac{1}{\sqrt{\timeStep}} \| p^\timeStep_{\level - 1} \tilde \lambda_{\level - 1} + \timeStep \Div \bar w \|_0 \\
  \le & \underbrace{ \| e_{\level - 1} \|_{a_{\level - 1}} }_{ =: \Xi_1 } + \underbrace{ \| \gradU^\timeStep_{\level - 1} \tilde \lambda_{\level - 1} - \nabla \bar w \|_0 }_{ =: \Xi_2 } + \underbrace{ \tfrac{1}{\sqrt{\timeStep}} \| p^\timeStep_{\level - 1} \tilde \lambda_{\level - 1} + \timeStep \Div \bar w \|_0 }_{ =: \Xi_3 }.
 \end{align}
 Next, we estimate the three individual terms
 \begin{align}
  \Xi^2_1 =~& a_{\level - 1}(e_{\level - 1}, e_{\level - 1}) = a_\level ( \lambda, \injectionOp_\level e_{\level - 1}) - a_{\level - 1}(\tilde \lambda_{\level - 1}, e_{\level - 1}) \\
  \overset{\eqref{EQ:scp_def}}{\underset{\eqref{EQ:tilde_lambda}}=~} & (f_\lambda, \liftingOp^\timeStep_\level \injectionOp_\level e_{\level - 1}) - (f_\lambda, u^\timeStep_{\level - 1} e_{\level - 1}) \\
  \le~ & \| f_\lambda \|_0 \| \liftingOp^\timeStep_\level \injectionOp_\level e_{\level - 1} - u^\timeStep_{\level - 1} e_{\level - 1} \|_0.
 \end{align}
 Note that
 \begin{align}
  \| \liftingOp^\timeStep_\level & \injectionOp_\level e_{\level - 1} - u^\timeStep_{\level - 1} e_{\level - 1} \|_0 \\
  & \le \| \liftingOp^\timeStep_\level \injectionOp_\level e_{\level - 1} - u^\timeStep_{\level} \injectionOp_\level e_{\level - 1} \|_0 + \| u^\timeStep_{\level} \injectionOp_\level e_{\level - 1} - u^\timeStep_{\level - 1} e_{\level - 1} \|_0 \\
  & = \| \liftingOp^\timeStep_\level \injectionOp_\level e_{\level - 1} - \Pi_\level \liftingOp^\timeStep_{\level} \injectionOp_\level e_{\level - 1} \|_0 + \| u^\timeStep_{\level} \injectionOp_\level e_{\level - 1} - u^\timeStep_{\level - 1} e_{\level - 1} \|_0 \\
  & \overset{\eqref{EQ:conv_u}}\lesssim h_\level \| \nabla \liftingOp^\timeStep_\level \injectionOp_\level e_{\level - 1} \|_0 + h_\level (1 + \timeStep) \| \gradU^\timeStep_{\level - 1} e_{\level - 1} \|_0 \\
  & \overset{Lemma\ \ref{LEM:norm_equiv_s}}{\underset{\eqref{EQ:conv_l}}\lesssim} h_\level (1 + \timeStep)^{3/2} \| \gradU^\timeStep_{\level - 1} e_{\level - 1} \|_0 \lesssim h_\level (1 + \timeStep)^{3/2} \Xi_1,
 \end{align}
 which implies that $\Xi_1 \lesssim (1 + \timeStep)^{3/2} h_\level \| A_\level \lambda \|_\level$.

 Next, taking $\bar w = \contLinProj_{\level - 1} \tilde u$ to be the $L^2$ projection of $\tilde u$ to $\linElementSpace_{\level - 1}$, and assuming that $(\gradU^\timeStep_{\level - 1}, \, u^\timeStep_{\level - 1}, \, p^\timeStep_{\level - 1})$ is the RT-H approximation of \eqref{EQ:ap_flambda}, we have that $\gradU^\timeStep_{\level - 1} = \gradU^\timeStep_{\level - 1} \tilde \lambda_{\level - 1} + \gradU^\timeStep_{\level - 1} f_\lambda$, $p^\timeStep_{\level - 1} = p^\timeStep_{\level - 1} \tilde \lambda_{\level - 1} + p^\timeStep_{\level - 1} f_\lambda$.
 \begin{align}
  \Xi_2 & \overset{\eqref{EQ:ap_flambda_1}}\le \| \gradU^\timeStep_{\level - 1} \tilde \lambda_{\level - 1} - \tilde \gradU \|_0 + \| \nabla \tilde u - \nabla \contLinProj_{\level - 1} \tilde u \|_0 \\
  & \le \| \gradU^\timeStep_{\level - 1} - \tilde \gradU \|_0 + \| \gradU^\timeStep_{\level - 1} f_\lambda \|_0 + \| \nabla \tilde u - \nabla \contLinProj_{\level - 1} \tilde u \|_0 \\&\lesssim(1 + \timeStep) C_\timeStep h_\level \| f_\lambda \|_0 + h_\level \| f_\lambda \|_0,
 \end{align}
 where the last inequality is a consequence of Theorem \ref{TH:error_auxiliary}, \eqref{EQ:stability_l_f}, and \eqref{EQ:regularity_est}.
 \begin{align}
  \Xi_3 & \le \tfrac{1}{\sqrt{\timeStep}} \left( \| p^\timeStep_{\level -1} \tilde \lambda_{\level -1} - \tilde p \|_0 + \timeStep \| \Div \tilde u - \Div \bar w \|_0 \right) \\
  & \lesssim \tfrac{1}{\sqrt{\timeStep}} \left( \| p^\timeStep_{\level - 1} -  \tilde p \|_0 + \|p^\timeStep_{\level - 1} f_\lambda \|_0 + \timeStep h_\level | \tilde u |_{2,\Omega} \right) \\
  & \lesssim (1 + \timeStep) C_\timeStep h_\level \| f_\lambda \|_0 + h_\level \| f_\lambda \|_0,
 \end{align}
 where the  last inequality is a combination of Theorem \ref{TH:error_auxiliary}, \eqref{EQ:error_a}, \eqref{EQ:stability_p_f}, and \eqref{EQ:regularity_est}.

 The combination of the estimates of $\Xi_1$, $\Xi_2$, and $\Xi_3$ with Lemma \ref{LEM:bound_flambda} finishes the proof.
\end{proof}

Using the aforementioned results, we can prove \eqref{EQ:a1} provided \eqref{EQ:ia1} and \eqref{EQ:ia2} hold.
\begin{theorem}
 \eqref{EQ:a1} holds.
\end{theorem}

\begin{proof}
 First, we prove the inequality
 \begin{equation}
  | a_\level (\lambda - \injectionOp_\level \projectionOp_{\level - 1}) \lambda, \lambda) | \le \bar C^2 h^2_\level \| A_\level \lambda \|^2_\level
 \end{equation}
 using binomial factorization and Lemma \ref{LEM:quasi_orthogonality}, we have
 \begin{align}
  a_\level (\lambda - & \injectionOp_\level \projectionOp_{\level - 1} \lambda, \lambda) = a_\level(\lambda, \lambda) - a_{\level - 1 } (\projectionOp_{\level - 1} \lambda, \projectionOp_{\level - 1} \lambda) \\
  =~ & (\gradU^\timeStep_\level \lambda, \gradU^\timeStep_\level \lambda) + \tfrac{1}{\timeStep} (p^\timeStep_\level \lambda, p^\timeStep_\level \lambda) \\
  & - (\gradU^\timeStep_{\level-1} \projectionOp_{\level - 1} \lambda, \gradU^\timeStep_{\level -1}\projectionOp_{\level - 1} \lambda) - \tfrac{1}{\timeStep} (p^\timeStep_{\level-1} \projectionOp_{\level - 1} \lambda, p^\timeStep_{\level-1} \projectionOp_{\level - 1} \lambda) \\
  =~ & (\gradU^\timeStep_\level \lambda + \gradU^\timeStep_{\level-1} \projectionOp_{\level - 1} \lambda, \gradU^\timeStep_\level \lambda - \gradU^\timeStep_{\level-1} \projectionOp_{\level - 1} \lambda) \\
  & + \tfrac{1}{\timeStep} (p^\timeStep_\level \lambda + p^\timeStep_{\level-1}\projectionOp_{\level - 1} \lambda, p^\timeStep_\level \lambda - p^\timeStep_{\level-1} \projectionOp_{\level - 1} \lambda) \\
  =~ & (\gradU^\timeStep_\level \lambda + \gradU^\timeStep_{\level-1} \projectionOp_{\level - 1} \lambda - 2 \nabla \bar w, \gradU^\timeStep_\level \lambda - \gradU^\timeStep_{\level-1} \projectionOp_{\level - 1} \lambda + \nabla \bar w -\nabla \bar w) \\
  & + \tfrac{1}{\timeStep} (p^\timeStep_\level \lambda + p^\timeStep_{\level-1} \projectionOp_{\level - 1} \lambda +2 \timeStep \Div \bar w, \\
  & \qquad \qquad p^\timeStep_\level - p^\timeStep_{\level-1} \projectionOp_{\level - 1} \lambda + \timeStep \Div \bar w - \timeStep \Div \bar w),
 \end{align}
 where $\bar w \in \linElementSpace_{\level - 1}$.
Then using Lemma \ref{EQ:reconstruction_approximation} and binomial factorization again,
 \begin{align}
  | a_\level(\lambda & - \injectionOp_\level \projectionOp_{\level - 1} \lambda, \lambda) | \\
  =~ & \| \gradU^\timeStep_\level \lambda - \nabla \bar w \|^2_0 - \| \gradU^\timeStep_{\level-1} \projectionOp_{\level - 1} \lambda - \nabla \bar w \|^2_0 \\
  & + \tfrac{1}{\timeStep} \left( \|p^\timeStep_\level \lambda + \timeStep \Div \bar w \|^2_0 - \| p^\timeStep_{\level-1} \projectionOp_{\level - 1} \lambda + \timeStep \Div \bar w \|^2_0 \right) \\
  \le~ & \| \gradU^\timeStep_\level \lambda - \nabla \bar w \|^2_0 + \tfrac{1}{\timeStep} \| p^\timeStep_\level \lambda + \timeStep \Div \bar w \|^2_0 \\
  &+\| \gradU^\timeStep_{\level-1} \projectionOp_{\level - 1} \lambda - \nabla \bar w \|^2_0+ \tfrac{1}{\timeStep} \| p^\timeStep_{\level-1} \projectionOp_{\level - 1} \lambda + \timeStep \Div \bar w \|^2_0\\
  \le~ & \bar C^2 h^2_\level \| A_\level \lambda \|^2_\level.
 \end{align}
 Thus, according to Lemma \ref{LEM:eigenvalue_bound}, \eqref{EQ:a1} holds with $C_1 = \bar C^2 (1 + \timeStep)$
\end{proof}
%
\section{Numerical experiments}
%
\begin{figure}
 \begin{tikzpicture}
  \draw[dashed] (0,0) -- (4,0);
  \draw (0,4) -- (4,0) -- (4,4) -- (0,0) -- (0,4) -- (4,4);
  \draw (2,0) -- (4,2) -- (2,4) -- (0,2) -- (2,0) -- (2,4);
  \draw (0,2) -- (4,2);
 \end{tikzpicture}
 \caption{Initial mesh for numerical experiments. The Neumann boundary $\Gamma_\textup N$ is dashed.}\label{FIG:initial_mesh}
\end{figure}
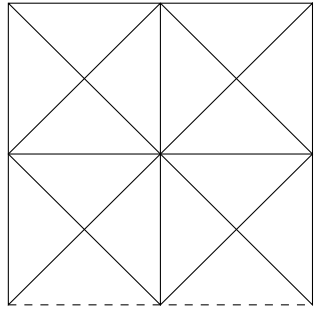
We consider the following Stokes problem on the unit square $\Omega = (0,1)^2$
\begin{subequations}\label{EQ:numerics}
\begin{align}
 - \Delta u + \nabla p & = f && \text{ in } \Omega,\\
 \nabla \cdot u & = 0 && \text{ in } \Omega,\\
 u & = u_\textup D && \text{ on } \Gamma_\textup D,\\
 (\nabla u + p I) \cdot \Nu & = g && \text{ on } \Gamma_\textup N,
\end{align}
\end{subequations}
where $\Gamma_\textup N = \{ x \in \partial \Omega\colon x_2 = 0\}$ and $\Gamma_\textup D = \partial \Omega \setminus \Gamma_\textup N$. We select $f$, $u_\textup D$ and $g$ such that
\begin{equation*}
 u = \begin{pmatrix} \sin(\pi x_1) \sin(\pi x_2) \\ \cos(\pi x_1) \cos(\pi x_2) \end{pmatrix} \qquad \text{ and } \qquad p = \sin(\pi x_1) \cos(\pi x_2)
\end{equation*}
solve \eqref{EQ:numerics}. For the method, the initial mesh is shown in Figure \ref{FIG:initial_mesh}, and  we set $\tau^\star_\level = 1$. Denoting the system of equations arising from \eqref{EQ:hdg_condensed} by $A \vec x = \vec b$ we choose the injection operator to be $\injectionOp^1_\level$ from \cite{LuRK21}. The iterative procedure for solving $A \vec x = \vec b$ is stopped if
\begin{equation*}
 \frac{\| A \vec x - \vec b \|_2}{\| \vec b \|_2} < \rho.
\end{equation*}

Importantly, we set $\epsilon_\textup{tol} = 10^{-8}$ in \eqref{EQ:sol_accuracy} and $\rho = 10^{-10}$ for linear (HDGP1) and quadratic (HDGP2) test and trial functions. If cubic (HDGP3) test and trial functions are used, we set $\epsilon_\textup{tol} = 10^{-10}$ and $\rho = 10^{-12}$.

First, we evaluate the depence of $n_\textup{iter}$ on the mesh level and the time-step size $\Delta t$ in Table \ref{TAB:n_iter}.
\begin{table}
 \begin{tabular}{c|ccccc|ccccc}
  \toprule
  mesh level & 1 & 2 & 3 & 4 & 5 & 1 & 2 & 3 & 4 & 5 \\
  \midrule
  $\Delta t = 2$ & 60 & 61 & 62 & 62 & 62 & 60 & 50 & 41 & 31 & 22 \\
  $\Delta t = 4$ & 35 & 36 & 36 & 36 & 36 & 35 & 30 & 25 & 19 & 14 \\
  $\Delta t = 8$ & 22 & 22 & 23 & 23 & 23 & 22 & 19 & 16 & 13 &  9 \\
  \bottomrule
 \end{tabular}\vspace{1ex}
 \caption{Necessary number of iterations $n_\textup{iter}$ for HDGP1 if $p^0_\level = 0$ (left) and if $p^0_1 = 0$, $p^0_\level = p^{n_\textup{iter}}_{\level-1}$ ($\level > 1)$.}\label{TAB:n_iter}
\end{table}
We can clearly see that $n_\textup{iter}$ decreases with increasing $\Delta t$ and when the results of the coarser mesh is used to initialize $p^0_\level$. In the remainder of this section, we use $p^0_1 = 0$ and $p^0_\level = p^{n_\textup{iter}}_{\level-1}$ if $\level > 1$.

\begin{table}
 \begin{tabular}{c|@{\,}lcc@{\,}lcc@{\,}lcc@{\,}lcc@{\,}lcc}
  \toprule
  mesh level  && \multicolumn{2}{c}{2}  && \multicolumn{2}{c}{3}   && \multicolumn{2}{c}{4}   && \multicolumn{2}{c}{5}    && \multicolumn{2}{c}{6} \\
  \# DoFs               && \multicolumn{2}{c}{368} && \multicolumn{2}{c}{1504} && \multicolumn{2}{c}{6080} && \multicolumn{2}{c}{24448} && \multicolumn{2}{c}{98048}  \\
  \cmidrule{3-4} \cmidrule{6-7} \cmidrule{9-10} \cmidrule{12-13} \cmidrule{15-16}
  smoother    && 2  & 4  && 2  & 4  && 2  & 4  && 2  & 4  && 2  & 4  \\
  \midrule
  $\Delta t = 2$ && 32 & 19 && 32 & 17 && 31 & 17 && 31 & 17 && 31 & 16 \\
  $\Delta t = 4$ && 43 & 25 && 46 & 24 && 45 & 24 && 44 & 24 && 43 & 23 \\
  $\Delta t = 8$ && 61 & 34 && 71 & 37 && 72 & 37 && -- & 37 && -- & 37 \\
  \bottomrule
 \end{tabular}\vspace{1ex}
 \caption{Numbers of iterations with two and four smoothing steps for HDGP1. The -- means that the iteration number has surpassed 100.}\label{TAB:steps_p1}
\end{table}
Table \ref{TAB:steps_p1} shows the iteration numbers with two and four smoothing steps for \eqref{EQ:hdg_condensed} and HDGP1. The iteration numbers for \eqref{EQ:sol_accuracy} are those of Table \ref{TAB:n_iter}. We clearly observe that the numbers of iterations increase with $\Delta t$, which coincides with our analysis. If we use four smoother steps, the iteration number is much smaller and more stable than for two smoothing steps. Thus, we restrict ourselves to four smoothing steps in the remainder of this section.

Table \ref{TAB:steps_p23} shows the numbers of iterations for \eqref{EQ:sol_accuracy} (columns A) and  \eqref{EQ:hdg_condensed} (columns B) with different choice of $\Delta t$ for HDGP2 and HDGP3 respectively. We can see that 
 the numbers in columns B increase with $\Delta t$, while for fixed $\Delta t$, the iteration steps are stable, which coincides with our analysis. 

 Let the estimated orders of convergence
(EOC) of the approximate to $u$ be evaluated by the formula
\begin{equation}
 \text{EOC} = \log \left( \frac{ \| u - u_{\level-1} \|_{L^2(\Omega)} }{ \| u - u_\level \|_{L^2(\Omega)} } \right) / \log (2).
\end{equation}
We illustrate the EOC of the discrete unknowns $p$, $u$ and $L$ for HDGP1, HDGP2 and HDGP3 respectively in Table \ref{TAB:eoc}.
These discrete unknowns of (\ref{EQ:stokes_mixed}) are computed using the augmented Lagrangian approach and the multigrid method
proposed in this paper, we can observe that the convergence order is as expected.
\begin{table}
 \begin{tabular}{cc|ccccc}
  \toprule
  \multicolumn{2}{c}{mesh level} & 2 & 3 & 4 & 5 & 6 \\
  \midrule
  \multirow{3}{*}{\rotatebox[origin=c]{90}{HDGP1}} 
  & $p$ & 2.35 & 2.23 & 2.12 & 2.06 & 2.03 \\
  & $u$ & 1.96 & 2.00 & 2.00 & 2.00 & 2.00 \\
  & $L$ & 1.92 & 1.97 & 1.99 & 2.00 & 2.00 \\
  \midrule
  \multirow{3}{*}{\rotatebox[origin=c]{90}{HDGP2}} 
  & $p$ & 3.50 & 3.39 & 3.25 & 3.15 & 3.08 \\
  & $u$ & 2.97 & 2.99 & 3.00 & 3.00 & 3.00 \\
  & $L$ & 2.95 & 2.98 & 3.00 & 3.00 & 3.00 \\
  \midrule
  \multirow{3}{*}{\rotatebox[origin=c]{90}{HDGP3}}
  & $p$ & 4.22 & 4.12 & 4.06 & 4.03 & 4.00 \\
  & $u$ & 3.97 & 3.99 & 4.00 & 4.00 & 4.00 \\
  & $L$ & 3.95 & 3.98 & 3.99 & 4.00 & 4.00 \\
  \bottomrule
 \end{tabular}\vspace{1ex}
 \caption{Estimated orders of convergence for the different unknowns and different orers of approximatin spaces.}\label{TAB:eoc}
\end{table}

\begin{table}
 \begin{tabular}{cc|@{\,}lcc@{\,}lcc@{\,}lcc@{\,}lcc@{\,}lcc}
  \toprule
  \multicolumn{2}{c}{mesh level}  && \multicolumn{2}{c}{2}  && \multicolumn{2}{c}{3}   && \multicolumn{2}{c}{4}   && \multicolumn{2}{c}{5}    && \multicolumn{2}{c}{6} \\
  \cmidrule{4-5} \cmidrule{7-8} \cmidrule{10-11} \cmidrule{13-14} \cmidrule{16-17}
  \multicolumn{2}{c}{equation}  && A & B  && A & B  && A & B  && A & B  && A & B \\
  \midrule
  \multirow{4}{*}{\rotatebox[origin=c]{90}{HDGP2}}
  & \# DoFs   && \multicolumn{2}{c}{552} && \multicolumn{2}{c}{2256} && \multicolumn{2}{c}{9120} && \multicolumn{2}{c}{36672} && \multicolumn{2}{c}{147072}  \\
  & $\Delta t = 2$ && 36 & 10 && 28 & 10 && 21 & 10 && 13 & 10 &&  7 & 10 \\
  & $\Delta t = 4$ && 22 & 13 && 18 & 14 && 13 & 14 &&  9 & 14 &&  6 & 15 \\
  & $\Delta t = 8$ && 14 & 17 && 12 & 24 &&  9 & 31 &&  7 & 34 &&  5 & 34 \\
  \midrule
  \multirow{4}{*}{\rotatebox[origin=c]{90}{HDGP3}}
  & \# DoFs   && \multicolumn{2}{c}{736} && \multicolumn{2}{c}{3008} && \multicolumn{2}{c}{12160} && \multicolumn{2}{c}{48896} && \multicolumn{2}{c}{196096}  \\
  & $\Delta t = 2$ && 43 & 16 && 29 & 17 && 17 & 17 && 10 & 17 &&  6 & 17 \\
  & $\Delta t = 4$ && 26 & 20 && 18 & 21 && 11 & 22 &&  7 & 22 &&  5 & 22 \\
  & $\Delta t = 8$ && 17 & 25 && 12 & 30 &&  8 & 32 &&  6 & 33 &&  4 & 33 \\
  \bottomrule
 \end{tabular}\vspace{1ex}
 \caption{Numbers of iterations for \eqref{EQ:sol_accuracy} (columns A) and  \eqref{EQ:hdg_condensed} (columns B) for HDGP2 and HDGP3.}\label{TAB:steps_p23}
\end{table}
\section{Conclusions}
We have devised and analyzed mutligrid methods for HDG discretizations of the Stokes problem. Our analysis relied on relations among the SFH, the RT-H, and the BDM-H discretizations and heavily exploited the augmented Lagrangian approach. Numerical findings illuminated that our estimates appear to be sharp. In future research endeavours, we will devise multigrid schemes that do not rely on the Lagrangian approach.
%
\section*{Acknowledgements}
%
This work is supported by the German Research Foundation under Germany's Excellence Strategy EXC 2181/1 - 390900948 (the Heidelberg STRUCTURES Excellence Cluster) and by the Academy of Finland's grant number 350101 \emph{Mathematical models and numerical methods for water management in soils}. P.~Lu has been supported by the Alexander von Humboldt Foundation.
\bibliographystyle{ARalpha}
\bibliography{MultigridStokes}
\end{document}